\documentclass[12pt]{amsart}
\usepackage{amsmath,amssymb,amsbsy,amsfonts,amsthm,latexsym,
                        amsopn,amstext,amsxtra,euscript,amscd,mathrsfs,color,bm,cite}
                       
\usepackage{float} 
\usepackage[english]{babel}
\usepackage{mathtools}
\usepackage{todonotes}
\usepackage{url}
\usepackage[colorlinks,linkcolor=blue,anchorcolor=blue,citecolor=blue,backref=page]{hyperref}
\usepackage{enumitem}

\def\le{\leqslant}
\def\ge{\geqslant}

\def\leq{\leqslant}
\def\geq{\geqslant}

\newtheorem{theorem}{Theorem}
\newtheorem{corollary}[theorem]{Corollary}
\newtheorem{lemma}[theorem]{Lemma}

\newtheorem{remark}[theorem]{\bf Remark}

\def \vol{\mathrm {Vol}}
\def \lcm{\mathrm {lcm}}

\newcommand\cL{\mathcal{L}}
\newcommand\cF{\mathcal{F}}
\newcommand\cW{\mathcal{W}}

\newcommand{\Z}{\mathbb{Z}}
\newcommand{\Q}{\mathbb{Q}}
\newcommand{\R}{\mathbb{R}}
\newcommand{\F}{\mathbb{F}}
\newcommand{\C}{\mathbb{C}}

\newcommand{\e}{\operatorname{e}}

\usepackage{etoolbox}

\numberwithin{equation}{section}
\numberwithin{theorem}{section}

\def\sE{\mathsf {E}}

\def\sT{\mathsf {T}}

\newcommand{\QQ}{\mathbb{Q}}
\newcommand{\Fq}{\mathbb{F}_q}

\def\cA{{\mathcal A}}
\def\cB{{\mathcal B}}
\def\cC{{\mathcal C}}
\def\cD{{\mathcal D}}

\def\cF{{\mathcal F}}

\def\cI{{\mathcal I}}
\def\cJ{{\mathcal J}}

\def\cL{{\mathcal L}}

\def\cN{{\mathcal N}}

\def\cU{{\mathcal U}}

\def\cW{{\mathcal W}}

\def\ssum{\mathop{\sum\ldots \sum}}

\def \balpha{\bm{\alpha}}
\def \bbeta{\bm{\beta}}

\def\ov\QQ{\overline{\QQ}}

\DeclareMathOperator{\supp}{supp}

\def \Gal{{\mathrm {Gal}}}
\def\e{\mathbf{e}}
\def\eq{{\mathbf{\,e}}_q}

\def\mand{\qquad\mbox{and}\qquad}

\usepackage{mathtools}
\usepackage{todonotes}
\usepackage[norefs,nocites]{refcheck}

\def\vec#1{\mathbf{#1}}

\def\({\left(}
\def\){\right)}

\def\rf#1{\left\lceil#1\right\rceil}

\begin{document}

\title[Energy bounds for modular roots]
{Energy bounds for modular roots and their applications}

 \author[B. Kerr] {Bryce Kerr}
\address{B.K.: School of Science, University of New South Wales.
 Canberra, ACT 2600, Australia}
\email{bryce.kerr89@gmail.com}

 \author[I. D. Shkredov]{Ilya D. Shkredov}
\address{I.D.S.: London Institute for Mathematical Sciences, 21 Albemarle St., London W1S 4BS, UK}
\email{ilya.shkredov@gmail.com}

 \author[I.~E.~Shparlinski]{Igor E. Shparlinski}
 \address{I.E.S.: School of Mathematics and Statistics, University of New South Wales.
 Sydney, NSW 2052, Australia}
 \email{igor.shparlinski@unsw.edu.au}
 
 \author[A. Zaharescu]{Alexandru Zaharescu} 
 \address{A.Z.: Department of Mathematics, University of Illinois at
Urbana-Cham\-paign  1409 West Green Street, Urbana, IL 61801, USA and 
Simon Stoilow Institute of Mathematics of the Romanian Academy, P.O. Box 1-764, RO-014700 Bucharest, Romania}
 \email{zaharesc@illinois.edu}
 
 \begin{abstract}  
We generalise and improve some recent bounds for additive energies of modular roots. Our arguments use a variety of techniques, including those from additive combinatorics, algebraic number theory and the geometry of numbers. We give applications of these results to new  
bounds on correlations between  {\it Sali{\'e}} sums and 
to a new equidistribution estimate for the set of modular roots of primes. 
\end{abstract}  


\maketitle

\tableofcontents

\section{Introduction}
\subsection{Background} 
For a prime $q$ we use $\Fq$ to denote the finite field of $q$ elements. 
Given a set $\cN \subseteq \Fq$ and an integer $k\ge 1$,  let $T_{\nu,k}(\cN;q)$ be  the number of solutions to 
the equation (in $\Fq)$, 
$$
b_1+\ldots+b_\nu=b_{\nu+1} +\ldots+b_{2\nu}, \qquad b_i^k \in \cN, \ i =1, \ldots, 2\nu.
$$
For $\nu=2$ we also denote 
$$
T_{\nu,k}(\cN;q) = E_{k}(\cN;q).
$$
When $k=1$,  this is the well-known in additive combinatorics quantity called the {\it additive energy\/} 
of $\cN$.  More generally, $ E_{k}(\cN;q)$ is the  additive energy of the set of $k$-th roots of elements of $\cN$ (of those which are $k$-th power residues). 

In the special case $\cN =  \{1, \ldots, N\}$ for an integer $1 \le N < q$,  we also write 
$$
T_{\nu,k}\(j\cN;q\) = \sT_{\nu,k}(N;j,q), \qquad E_{k}\( j\cN;q\)= \sE_k(N;j,q),
$$
where the set $j\cN = \{j, \ldots, jN\}$ is embedded in $\Fq$ in a natural way. 

The quantity $\sE_{2}(N;j,q)$ has been introduced and  estimated in~\cite{DKSZ}.
In particular, for any $j \in \Fq^*$, by~\cite[Lemmas~6.4 and~6.6]{DKSZ} we have 
\begin{equation}\label{eq:Energy-DKSZ}
 \sE_2(N;j,q)\le   \min\left\{N^4/q + N^{5/2}, \,  N^{7/2}/q^{1/2}+ N^{7/3}   \right\} q^{o(1)},
\end{equation}
which has been used in~\cite[Theorem~1.7]{DKSZ} to estimate certain bilinear 
sums and thus improve some results of~\cite{DuZa} on correlations between  {\it Sali{\'e}} sums,
which is important  for applications to moments of $L$-functions attached to some  modular forms.
Furthermore, bounds of such bilinear sums have applications to the distribution of modular square 
roots of primes, see~\cite{DKSZ, SSZ2} for details.

This line of research has been continued  in~\cite{SSZ1} where it is shown that for almost all primes $q$, 
for all  $N < q$ and  $j \in \Fq^*$ one has an essentially 
optimal bound
\begin{equation}\label{eq:Energy-SSZ-N}
 \sE_2(N;j,q)\le \(N^4/q + N^{2}\)q^{o(1)}.
\end{equation}
We expect the bound~\eqref{eq:Energy-SSZ-N} to hold for all primes $q$, however this seems difficult to establish with current techniques.

As an application of the bound~\eqref{eq:Energy-SSZ-N}, it has been show in~\cite{SSZ1} that on 
average over $q$ one can significantly improve the error term in the asymptotic 
formula for twisted second moments of $L$-functions of half integral weight modular forms.

Furthermore, it is shown  in~\cite{SSZ1}  that  methods of {\it additive combinatorics\/}  can 
be used  to estimate $ E_{2}(\cN;q)$ for sets $\cN$ with small doubling. 
Namely, for an arbitrary  set $\cN$ (of any algebraic domain equipped with addition), as usual, we denote
$$
\cN+\cN = \{n_1+n_2:~n_1,n_2 \in \cN\}.
$$
Then it is shown in~\cite{SSZ1}, in particular, that  if $\cN \subseteq \Z_q$ is a set of cardinality $N$ 
	such that   $\#\(\cN + \cN\)   \le LN$ for some real $L$, then 
\begin{equation}\label{eq:Energy-SSZ-Set-N}
 E_{2}(\cN;q) \le  q^{o(1)} \( \frac{ L^4 N^4}{q} + 
  L^2  N^{11/4}\) \,. 
\end{equation}

Here we extend and improve these results in several directions and  obtain upper bounds on 
$T_{\nu,k}(\cN;q)$ and $\sT_{\nu,k}(N;j,q)$ 
for other choices of $(\nu,k)$ besides $(\nu,k) = (2,2)$ along with improving the bound of~\cite[Lemma~6.6]{DKSZ} for $T_{2,2}(N;j,q)$. Our estimate for $T_{2,2}(N;j,q)$ gives some improvement on exponential sums bounds from~\cite{DKSZ}. 

We believe the new ideas of this work include
 \begin{itemize}
\item the use of higher dimensional lattices and more advanced techniques from the geometry of numbers such as transference principles and should be considered a development of the arguments from~\cite{DKSZ} where only a two dimensional lattice is used,
\item applying so-called {\it decimations\/} of multivariate polynomials,
\item the use of {\it Gowers norms. \/}
\end{itemize}
 Such estimates have the potential for several new applications. 
One such application  is to bilinear sums
with some {\it multidimensional Sali{\'e} sum\/} which by a result of Duke~\cite{Duke} 
can be reduced to one dimensional sums over $k$-th roots (generalising the case of 
$k=2$, see~\cite[Lemma~12.4]{IwKow} or~\cite[Lemma~4.4]{Sarn}). This result of 
 Duke~\cite{Duke}  combined with our present results and also the approach of~\cite{DuZa,DKSZ,SSZ1}, 
may have a potential to 
 lead to new asymptotic formulas for moments of $L$-functions with Fourier 
 coefficients of automorphic forms over ${\mathrm{GL}}(k, \Z)$ with $k \ge 3$, we refer to~\cite{Duke} 
 for further references.   
For these applications one has to extend our bound from $k=2$ to arbitrary $k \ge 3$, 
which is of independent interest, and maybe achievable with our techniques.

Improved bounds on $\sT_{\nu,k}(N;j,q)$ with $\nu > 2$ have a potential to obtain further improvements and extend the region in 
which there are non-trivial  bounds of bilinear sums from~\cite{DKSZ,SSZ1}. In turn this can lead to further advances in their applications.

 Furthermore, the new result on the distribution of modular roots of primes, 
 see Theorem~\ref{thm:GammaqP}, can be viewed as  dual to celebrated result
  of  Duke, Friedlander and Iwaniec~\cite{DFI1, DFI2} on square roots of a fixed integer
  modulo distinct primes. In turn, this may have a similar range of ``dual'' applications.

\subsection{Notation} 
Throughout the paper, the notation $U = O(V)$, 
$U \ll V$ and $ V\gg U$  are equivalent to $|U|\leqslant c V$ for some positive constant $c$, 
which throughout the paper may depend on the integer $k$. 

For any quantity $V> 1$ we write $U = V^{o(1)}$ (as $V \to \infty$) to indicate a function of $V$ which 
satisfies $|U| \le V^{\varepsilon}$ for any $\varepsilon> 0$, provided $V$ is large enough. 

 For complex weights  $\bbeta = \{\beta_n\}_{n\in \cN}$, supported on a finite set $\cN$, we  define the norms 
 $$
 \|\bbeta\|_\infty=\max_{n\in \cN}|\beta_n|  \mand \|\bbeta\|_\sigma =\( \sum_{n\in \cN} |\alpha_n|^\sigma\)^{1/\sigma},
 $$
where  $\sigma >1$, and similarly for other weights.

For a real $A> 0$, we write $a \sim A$ to indicate  that $a$ is in the dyadic interval $A/2 \le  a <  A$.

We use $\# \cA$ for the cardinality of a finite set $\cA$.

Given two functions $f,g$ on some algebraic domain $\cD$ equipped with addition, we define the convolution
$$(f\circ g)(d)=\sum_{x\in \cD}f(x)g(x-d).$$
We can then recursively define longer convolutions $(f_1\circ \ldots \circ f_s)(d)$.  

If $f$ is the indicator function of a set $\cA$ then we write
$$(f\circ f)(d)=\(\cA \circ \cA\) (d).$$ 

In fact, we often use $\cA(a)$ for the indicator function of a set $\cA$, that is, $\cA(a) = 1$ 
if $a \in \cA$ and $\cA(a) = 0$ otherwise. 

Note that $\(\cA \circ \cA\) (d)$ counts the number of the solutions to the equation $d=a_1-a_2$, where $a_1$, $a_2$ run over $\cA$, that is
\begin{equation}
\label{eq:A-circ}
\(\cA \circ \cA\) (d) = \# \{(a_1,a_2)\in \cA^2:~d=a_1-a_2\}. 
\end{equation}

  As usual, we also write 
$$
\cA + \cA = \{a_1+a_2:~a_1,a_2\in \cA\}
$$
and more generally
$$
k\cA -\ell \cA =   \{a_1+\ldots + a_k - b_1-\ldots-b_\ell:~a_1,\ldots , a_k,b_1,\ldots,b_\ell\in \cA\}.
$$

Finally, we follow the convention that in summation symbols $\sum_{a \le A}$ the sum 
is over positive integers $a \le A$.

\subsection{New results}
We start with a new bound  on $\sT_{2,2}(N;j,q)=\sE_2(N;j,q)$ which improves~\eqref{eq:Energy-DKSZ}. 

\begin{theorem}
\label{thm:T22-N}
Let $q$ be prime. For any $j\in \Fq^{*}$ and integer $N\le q$ we have
$$
\sT_{2,2}(N;j,q)\ll \left(\frac{N^{3/2}}{q^{1/2}}+1\right)N^{2+o(1)}.
$$
\end{theorem}

Note  it is easy to show the following trivial inequality 
$$
\sT_{4,2}(N;j,q)\le N^{4}\sT_{2,2}(N;j,q),
$$
which combined with Theorem~\ref{thm:T22-N} implies that
\begin{equation}
\label{eq:T42123}
\sT_{4,2}(N;j,q)\le \left(\frac{N^{3/2}}{q^{1/2}}+1\right)N^{6+o(1)}.
\end{equation}
We now obtain a stronger bound for short intervals. 

\begin{theorem}
\label{thm:T42-N}
Let $q$ be prime. For any $j\in \Fq^{*}$ and integer $N\le q$ we have
$$
\sT_{4,2}(N;j,q) \le \left(\frac{N^{5/8}}{q^{1/8}}+\frac{N^{11/2 }}{q^{1/2}}+\frac{N^{3}}{q^{1/4}}\right)N^{6+o(1)}+N^{5+o(1)}.
$$ 
\end{theorem}  
 We see that Theorem~\ref{thm:T42-N} is sharper than~\eqref{eq:T42123} provided $N\le q^{1/12}$.  Energies of the type considered in Theorem~\ref{thm:T42-N} have the potential for applications to new bilinear sum estimates considered in Section~\ref{sec:appl} below. However, the range of parameters $N\le q^{1/12}$ does not seem strong enough for meaningful applications, except maybe to very skewed bilinear sums.

The  proofs of Theorems~\ref{thm:T22-N} and~\ref{thm:T42-N} are based on the geometry of 
numbers and in particular on some properties of lattices. Although such ideas have been used before 
to estimate the number of solutions of various congruences, see~\cite{BGKS,KerrMoh}, they have never been applied to estimate the additive energy of modular roots. 

Next we generalise~\eqref{eq:Energy-SSZ-N}  to higher order roots.  
In fact, as in~\cite{SSZ1} the methods allow us to 
also treat the natural extension of $\sE_{k} (N;j,q)$ to composite moduli $q$,  for which we consider equations in the
residue ring $\Z_q$ modulo $q$, and estimate  $\sE_{k} (N;j,q)$  for almost all positive integers $q$. 
We however restrict ourselves to the case of prime moduli $q$. 

\begin{theorem}
\label{thm:E2k} 
For a fixed $k \ge 3$ and any  positive integers $Q \ge  N \ge 1$, we have 
$$
  \frac{\log Q}{Q}   \sum_{\substack{ q \sim Q \\ q~\text{prime}}}  \max_{j \in \F_q^*} 
\sE_{k} (N;j,q)  \ll   N^2  + N^{4}  Q^{-1+o(1)} . 
$$
\end{theorem}

To establish Theorem~\ref{thm:E2k} we use some arguments related to norms of algebraic integers. It is interesting to note that  our construction of auxiliary polynomials resemble the so-called {\it decimation\/} procedure which appears
in multiple contexts, we refer to~\cite{ALSV} for further references.  
 
We now extend the bound~\eqref{eq:Energy-SSZ-Set-N} to other values of $k$ as follows.

\begin{theorem}
\label{thm:E2k-Set} 
	Let  $\cN \subseteq \Fq$  be a set of cardinality $\# \cN = N \le q^{2/3}$ 
	such that   $\#\(\cN + \cN\)   \le LN$ for some real $L$. 
Then for $k \ge 3$ we have 
$$
 E_{k}(\cN;q)\le   L^{\vartheta_k} N^{3-\rho_k } q^{o(1)} \,, 
$$
where
$$
 \rho_k = 1/(7\cdot 2^{k-1}-9) \quad \text{and}\quad 
\vartheta_k = \begin{cases}
2^{k+3} \rho_k, & \text{ for $k\ge 5$};\\
64/47, & \text{ for $k=4$};\\ 
32/19, & \text{ for $k=3$}.
\end{cases} 
$$
\end{theorem}

We remark that the exponent of $L$ in Theorem~\ref{thm:E2k-Set} is $\vartheta_3 = 32/19$,  $\vartheta_4 = 64/47$
and  
$$
\vartheta_k = \frac{2^{k+3}}{ 7\cdot 2^{k-1}-9}  \le \frac{256}{103} 
$$
for $k \ge 5$. For $k=3, 4$ the exponent of $L$ is better than generic because 
of some additional saving  in our application 
of the {Pl\"unnecke inequality\/}, see~\cite[Corollary~6.29]{TaoVu}. 

The proof is based on some ideas of Gowers~\cite{Gow_4, Gow_m}, in particular on the 
notion of the {\it Gowers norm\/}. 
Finally, we remark that it is easy to see that, actually, our method works for any polynomial not only for  monomials. 
Also, it is possible, in principle, to insert the general weight $\bbeta$ but the induction procedure requires complex  calculations to estimate this more general quantity 
$$
E_{k}(\cN;\bbeta, q) = \sum_{\substack{u,v,x,y \in \F_q\\
u^k, v^k, x^k, y^k \in\cN\\u+v=x+y}}\beta_u\beta_v \beta_x \beta_y.
$$
Nevertheless, we record a simple consequence of Theorem~\ref{thm:E2k-Set} with weights $\bbeta$,
 which follows from the pigeonhole principle.

\begin{corollary} 
	Let  $\cN \subseteq \Fq$  be a set of cardinality $\# \cN = N$ 
	such that   $\#\(\cN + \cN\)   \le LN$ for some real $L$. 
	Then for any weights  $\bbeta$  
	supported on $\cN$,  and with $\| \bbeta\|_\infty \le 1$ 	
	Then 
\[
 E_{k}(\cN;\bbeta, q)\le    L^{\vartheta_k}
    \| \bbeta\|^{2-2\rho_k}_1 
    \| \bbeta\|^{2+2\rho_k}_2  q^{o(1)}
    \,,
\]	
where $\vartheta_k$ and $\rho_k$ are as in Theorem~\ref{thm:E2k-Set}.
\end{corollary}

We also remark that  Theorem~\ref{thm:E2k-Set} can be reformulated as a statement that for any set $\cA \subseteq \F_q$ either the additive energy  $\#\{a_1+a_2 = a_3+a_4:~a_1, a_2 ,a_3, a_4\in \cA\}$
 of $\cA$ is small or $\cA^k$ has  large doubling set $\cA^k + \cA^k =\{a_1^k+a_2^k:~a_1, a_2  \in \cA\}$.

\section{Applications}
\label{sec:appl}
Given weights $\balpha,\bbeta$ and $a,h\in \F_q^{*}$ we define bilinear forms over modular square roots as in~\cite[Equation~(1.6)]{DKSZ}
\begin{equation} \label{Wdef1}
W_{a,q}(\balpha, \bbeta; h,M,N)=   \sum_{m \sim M}  \sum_{n \sim N} \alpha_m \beta_n \sum_{\substack{x \in \Fq \\
x^2 = amn}} \eq(hx).
\end{equation} 
Using Theorem~\ref{thm:T22-N} we obtain a new estimate for $W_{a,q}(\balpha, \bbeta; h,M,N)$ which improves on~\cite[Theorem~1.7]{DKSZ}. Assuming 
$$
\|\balpha\|_{\infty}, \|\bbeta\|_{\infty} \le 1,
$$
it follows from the proof of~\cite[Theorem~1.7]{DKSZ} that 
$$
|W_{a,q}(\balpha, \bbeta; h,M,N)|^8\le q^{1+o(1)}(MN)^{4}\sT_{2,2}(M;1,q) \sT_{2,2}(N;b,q) ,
$$
for some $b$ with $\gcd(b,q)=1$. 

Applying Theorem~\ref{thm:T22-N}, we obtain the following bound. 

\begin{corollary} \label{cor:Waq}
For any positive integers $M,N\le q/2$ and any weights $\balpha$  and $\bbeta$ satisfying 
$$
\|\balpha\|_{\infty}, \|\bbeta\|_{\infty} \le 1,
$$
we have 
$$
 |W_{a,q}(\balpha,\bbeta;h,M,N)|\le q^{1/8+o(1)}(MN)^{3/4}\left(\frac{M^{3/16}}{q^{1/16}}+1\right) \left(\frac{N^{3/16}}{q^{1/16}}+1\right). 
$$
\end{corollary}

If the sequence $\bbeta$ corresponds to values of a smooth function $\varphi$ whose derivatives and support $\supp \varphi$ satisfy
\begin{equation}
\label{eq:cond phi}
\varphi^{(j)}(x)\ll \frac{1}{x^{j}} \mand  \supp \varphi \subseteq [N,2N],
\end{equation}
for any integer $j$ (with implied constant allowed to depend on $j$) then we write
\begin{equation}
\label{eq:Vaq}
V_{a,q}(\balpha, \varphi;h,M,N)=\sum_{m\sim M}\sum_{n\in \Z}\alpha_m\varphi(n)\sum_{\substack{x\in \Fq \\ x^2=amn}}\eq(hx).
\end{equation}

We now give a new bound for $V_{a,q}(\balpha;h,M,N)$. This does not rely on energy estimates although may be of independent interest. 
It is also used in a combination with Corollary~\ref{cor:Waq}  to derive Theorem~\ref{thm:GammaqP}  below.

\begin{theorem} \label{thm:b1}
For any positive integers $M,N$ satisfying $MN \ll q$ and $M<N$, any weight $\balpha$ satisfying 
$$
\|\balpha\|_{\infty} \le 1,
$$
and a function  $\varphi$  satisfying~\eqref{eq:cond phi}, for any fixed integer $r\ge 2$,  we have 
$$
|V_{a,q}(\balpha, \varphi;h,M,N)| \le
 q^{1/2-1/4r+o(1)}M^{1-1/2r}N^{1/2r}\(1+\frac{(MN)^{1/2}}{q^{1/2-1/4r}}\).
$$
\end{theorem}

Corollary~\ref{cor:Waq} may be used to improve various results from~\cite[Sections~1.3--1.4]{DKSZ}. We present once such improvement to the distribution of modular roots of primes.  Recall that  the {\it discrepancy\/} $D(N) $ of a sequence in $\xi_1, \ldots, \xi_N \in [0,1)$    is defined as 
$$
D_N = 
\sup_{0\le \alpha < \beta \le 1} \left |  \#\{1\le n\le N:~\xi_n\in [\alpha, \beta)\} -(\beta-\alpha)  N \right |.
$$

For a positive integer $P$  we denote the discrepancy of the sequence (multiset) of points
$$
 \{x/q:~ x^2 \equiv p \bmod q \ \text{for some prime} \ p\le P\}
$$
by $\Gamma_q(P)$.
 Combining the Erd\"{o}s-Tur\'{a}n inequality with the Heath-Brown identity reduces estimating $\Gamma_q(P)$ to sums of the form~\eqref{Wdef1} and~\eqref{eq:Vaq}. 
Combining, Corollary~\ref{cor:Waq}  with Theorem~\ref{thm:b1}, 
we obtain an improvement on~\cite[Theorem~1.10]{DKSZ}.
 
\begin{theorem}
\label{thm:GammaqP}
For any $P\le q^{3/4}$ we have 
$$
\Gamma_q(P)\le  \(q^{1/16}P^{7/8}+q^{13/88}P^{3/4}\)q^{o(1)}. 
$$
\end{theorem}

Note that Theorem~\ref{thm:GammaqP} is nontrivial provided $P\ge q^{13/22}$ and improves on the range $P\ge q^{13/20}$ from~\cite[Theorem~1.10]{DKSZ}.

\section{Proof of Theorem~\ref{thm:T22-N}}

\subsection{Lattices}

 We use $\vol(B)$ to denote the volume of a body $B \subseteq \R^d$. For a lattice $\Gamma\subseteq \R^{d}$ we recall that the quotient space  $\R^d/\Gamma$
 (called the fundamental domain) is compact and so $\vol(\R^d/\Gamma)$ is correctly defined, see also~\cite[Sections~3.1 and~3.5]{TaoVu}
 for basic definitions and properties of lattices.
 In particular, we  define the successive minima $\lambda_i$, $i=1, \ldots, d$,  of $B$ with respect to $\Gamma$ as
 $$
\lambda_i =\inf\{\lambda >0:~\lambda B \text{ contains $i$ linearly independent elements of } \Gamma\},
$$
where $\lambda B$ is the homothetic image of $B$ with the coefficient $\lambda$.

  The following is Minkowski's second theorem, for a proof see~\cite[Theorem~3.30]{TaoVu}. 
  
 \begin{lemma}
\label{lem:mst}
Suppose $\Gamma \subseteq \R^{d}$ is a lattice of rank $d$, $B\subseteq \R^{d}$ a symmetric convex body and let $\lambda_1,\ldots,\lambda_d$ denote the successive minima of $\Gamma$ with respect to $B$. Then we have
$$
\frac{1}{\lambda_1\ldots\lambda_d}\le \frac{d!}{2^d}\frac{ \vol (B)}{\vol(\R^d/\Gamma)}.
$$
\end{lemma}

 For a proof of the following,  see~\cite[Proposition~2.1]{BHM}.
\begin{lemma}
\label{lem:latticesm}
Suppose $\Gamma \subseteq \R^{d}$ is a lattice, $B\subseteq \R^{d}$ a symmetric convex body and let $\lambda_1,\ldots,\lambda_d$ denote the successive minima of $\Gamma$ with respect to $B$. Then we have
$$\# \(\Gamma \cap B\) \le \prod_{i=1}^{d}\left(\frac{2i}{\lambda_i}+1\right).$$
\end{lemma}

\subsection{Reduction to counting points in lattices}
\label{sec:RedLatt} 
It more convenient to  estimate  $\sT_{2,2}(N;\overline{j},q)$  rather than 
 $\sT_{2,2}(N;j,q)$
for the multiplicative inverse 
$\overline{j}$ of $j$ modulo $q$, which or course is an equivalent question. 

Let $\cA$ denote the set 
$$
\cA=\{ x\in \F_q^{*} :~ j x^2 \in \{1,\ldots, N\} \},
$$
so that 
\begin{equation}
\label{eq:EAI}
\sT_{2,2}(N;\overline{j},q)=\sum_{d\in \F_q}(\cA\circ \cA)(d)^2.
\end{equation}
where $(\cA\circ \cA)(d)$ is defined by~\eqref{eq:A-circ}. 

If $a_1,a_2\in \cA$ satisfy
$$a_1-a_2=d,$$
then elementary algebraic manipulations imply
$$
(a_1^2-a_2^2-d^2)^2=4d^2a_2^2.
$$
We have
$$ja_1^2-ja_2^2, ja_2^2\in \{-N,\ldots,N\}.$$
Since for any $\lambda, \mu\in \F_q$ the number of solutions to
$$ja_1^2-ja_2^2=\lambda, \quad  ja_2^2=\mu, \qquad a_1,a_2\in \cA,$$
is $O(1)$, we derive from~\eqref{eq:EAI}
$$
\sT_{2,2}(N;\overline{j},q) \ll \sum_{d\in \F_q}J_0(d)^2,
$$
where
$$
J_0(d)=\# \{ |m|,|n|\le N :~ (n-jd^2)^2 \equiv 4jd^2m \bmod{q} \}.
$$
 If $n,m$ satisfy 
$$
(n-jd^2)^2 \equiv 4jd^2m \bmod{q},
$$
then 
$$
n^2+j^2d^4\equiv 2jd^2(2m+n) \bmod{q}.
$$
This implies 
\begin{equation}
\label{eq:EaJ}
\sT_{2,2}(N;\overline{j},q)\ll \sum_{d\in \F_q}J(d)^2,
\end{equation}
where
\begin{equation}
\label{eq:J1def}
J(d)=\# \{ |m|,|n|\le 6N :~ n^2+j^2d^4\equiv jd^2m \bmod{q} \}.
\end{equation}
Let $\cL(d)$ denote the lattice 
$$
\cL(d)=\{ (x,y) \in \Z^2 :~ x\equiv jd^2 y \bmod{q} \},
$$
 $B$ the convex body 
$$
B=\{(x,y)\in \R^2 :~ |x|\le 72 N^2,\  |y|\le 12N \},
$$
and let $\lambda_1(d),\lambda_2(d)$ denote the first and second successive minima of $\cL(d)$ with respect to $B$. 

We now partition summation in~\eqref{eq:EaJ} according to the size of $\lambda_1(d)$ and 
$\lambda_2(d)$ to get 
\begin{equation}
\label{eq:S0123}
\sT_{2,2}(N;\overline{j},q)\ll S_0+S_1+S_2,
\end{equation}
where 
$$
S_0 =\sum_{\substack{d\in \F_q \\ \lambda_1(d)>1}}J(d)^2,\qquad
S_1  =\sum_{\substack{d\in \F_q \\ \lambda_1(d)\le 1 \\ \lambda_2(d)>1}}J(d)^2,\qquad 
S_2 =\sum_{\substack{d\in \F_q \\ \lambda_1(d),\lambda_2(d)\le 1}}J(d)^2.
$$
\subsection{Concluding the proof}
Consider first $S_0$. If $\lambda_1(d)>1$ then 
$$J(d)\le 1,$$
which follows from the fact that for any distinct points $(n_0,m_0)$, $(n_1.m_1)$ satisfying the conditions in~\eqref{eq:J1def} we have 
$$(n_0^2-n_1^2,m_0-m_1)\in \cL(d)\cap B.$$
This implies that $J(d)^2 = J(d)$ and we derive
\begin{equation}
\label{eq:S0}
S_0=  \sum_{\substack{d\in \F_q \\ \lambda_1(d)>1}}J(d)\ll N^2.
\end{equation} 

Consider next $S_1$. Suppose $d$ satisfies $\lambda_1(d)\le 1$ and $\lambda_2(d)> 1$. There exists $n_d,m_d$ satisfying the conditions given in~\eqref{eq:J1def} such that
$$
J(d)\ll \# \left\{ |m|,|n|\le 6N :~ (n^2-n_d^2,m-m_d)\in \cL(d)\cap B\right\} .
$$
Since $\lambda_2(d)>1$, there exists a unique point $(a_d,b_d) \in   \cL(d)\cap B$ satisfying 
$$
\gcd(a_d,b_d)=1, \quad |a_d|\le 72N^2, \quad |b_d|\le 12N,
$$
such that 
$$
J(d)\ll \# \left\{ |m|,|n|\le 6N :~ \frac{n^2-n_d^2}{m-m_d}=\frac{a_d}{b_d}\right\}+1  .
$$ 
This implies 
\begin{equation}
\begin{split} 
\label{eq:S1b1}
S_1&\le \sum_{d\in \F_q}J(d)\(\# \left\{ |m|,|n|\le 6N :~ \frac{n^2-n^2_d}{m-m_d}=\frac{a_d}{b_d}\right\}+1 \) \\
&\le \sum_{(a,b) \in \cW} J(a,b)K(a,b)+N^2, 
\end{split} 
\end{equation}
where $\cW$ is the following set of all pairs $(a,b)$ satisfying  
\begin{equation}
\label{eq:Set W}
\cW = \{(a,b)\in \Z^2:~|a|\le 72N^2, \  |b|\le 12 N, \  \gcd(a,b)=1\},
\end{equation}
and $K(a,b)$ is defined by
 $$
K(a,b)=\# \left\{(m,n) \in \Z^2:~|m|,|n|\le 6N, \ \frac{n^2-n_{a,b}^2}{m-m_{a,b}}=\frac{a}{b}\right\},
$$
for some choice of integers $m_{a,b},n_{a,b}$ satisfying $|m_{a,b}|,|n_{a,b}|\le 6 N$ and $J(a,b)$ is defined by 
$$
J(a,b)=\# \left\{(m,n) \in \Z^2:~ |m|,|n|\le 6N, \ n^2+(ab^{-1})^2\equiv ab^{-1}m \mod{q}\right\}.
$$

Note that
\begin{align*}
\sum_{(a,b)\in \cW}J(a,b)&\le \#\{(m,n,\lambda) \in \Z^3:~  |m|,|n|\le 6N, \ 1\le \lambda<q, \\
& \qquad \qquad \qquad\qquad \qquad\quad  \lambda^2-\lambda m+n^2 \equiv 0 \mod{q} \} \\ &\ll N^2
\end{align*}
 since after fixing $m,n$ with $O(N^2)$ choices, there exists $O(1)$ solutions to
$$\lambda^2-\lambda m+n^2 \equiv 0 \mod{q}$$
in the remaining variable $\lambda$. We also have 
\begin{equation}
\label{eq:Jab}
J(a,b)\ll K(a,b)+1.
\end{equation}
Fix some $a,b$ as in the  sum in~\eqref{eq:S1b1} and consider $K(a,b)$. If $n,m$ satisfy 
$$
\frac{n^2-n_{a,b}^2}{m-m_{a,b}}=\frac{a}{b}, \qquad |m|,|n|\le 6 N,
$$
then, since  $\gcd(a,b)=1$, we have 
\begin{equation}
\label{eq:aconds}
n^2- n_{a,b}^2 \equiv 0 \bmod{|a|},
\end{equation}
and 
\begin{equation}
\label{eq:bconds}
m-m_{a,b}\equiv 0 \bmod{|b|}.
\end{equation}

Furthermore, if one out of $m$ or $n$ is fixed then the other 
number is   defined in no more than two ways.

Write~\eqref{eq:aconds} as 
$$
(n-n_{a,b})(n+n_{a,b})\equiv 0 \bmod{|a|}.
$$
Then we see that there are two integers $a_1,a_2$ satisfying 
$$
a_1a_2=a, \qquad |a_1|,|a_2|\le 12 N,
$$  
such that 
$$
n\equiv n_{a,b} \bmod{|a_1|}, \quad n\equiv -n_{a,b} \bmod{|a_2|}.
$$
Hence for each fixed pair $(a_1, a_2)$ there are at most 
$$
\frac{N}{\lcm[a_1,a_2]}+1 \ll \frac{N}{|a|} \gcd(a_1,a_2) + 1.
$$
possibilities for $n$.  Hence by a well-known bound
\begin{equation}
\label{eq:div func}
\tau(a) = a^{o(1)}
\end{equation}
 on  the divisor function $\tau(a)$ for $a \ne 0$, see~\cite[Equation~(1.81)]{IwKow}, we have
$$
K(a,b) 
 \ll \sum_{a_1a_2=a}\(\frac{N}{\text{lcm}(a_1,a_2)}+1\)\ll \frac{N}{|a|}\sum_{a_1a_2=a}\gcd(a_1,a_2) + N^{o(1)} .
$$
By the Cauchy-Schwarz inequality and~\eqref{eq:div func},  
we now derive 
\begin{equation}
\label{eq:K_a-bound}
K(a,b)^2\ll N^{2+o(1)}\sum_{a_1a_2=a}\frac{\gcd(a_1,a_2)^2}{|a|^2} + N^{o(1)}. 
\end{equation}

Similarly, using~\eqref{eq:bconds} we obtain 
\begin{equation}
\label{eq:K_b-bound}
K(a,b)\ll \frac{N}{|b|}.
\end{equation}

Combining~\eqref{eq:K_a-bound},~\eqref{eq:K_b-bound},~\eqref{eq:Jab} and substituting into~\eqref{eq:S1b1}, we see that 
\begin{align*}
S_1\le N^{2+o(1)}\sum_{(a,b) \in \cW}\sum_{\substack{a_1a_2=a \\ |a_1|,|a_2|\le 12 N}}\min\left\{\frac{1}{b^2},\frac{\gcd(a_1,a_2)^2}{a^2} \right\} & \\
  + \sum_{(a,b)\in \cW}J(a,b) & N^{o(1)}. 
\end{align*} 
Hence, recalling~\eqref{eq:Set W}, we derive
\begin{align*}
S_1
&\le N^{2+o(1)}\sum_{\substack{|a|\le 72 N^2\\|b|\le 12 N}}\sum_{\substack{a_1a_2=a \\ |a_1|,|a_2|\le 12 N}}\min\left\{\frac{1}{b^2},\frac{\gcd(a_1,a_2)^2}{a^2} \right\} +  N^{2+o(1)}\\ 
& \le N^{2+o(1)}\sum_{ a_1,a_2, b \le 12 N}\min\left\{\frac{1}{b^2},\frac{\gcd(a_1,a_2)^2}{a^2_1a^2_2} \right\} +  N^{2+o(1)}\\ 
&\le N^{2+o(1)}\sum_{e \le 12N}\sum_{ b\le 12N}\sum_{\substack{a_1,a_2\le 12N \\ \gcd(a_1,a_2)=e}}
\min\left\{\frac{1}{b^2},\frac{e^2}{a^2_1a^2_2} \right\} +  N^{2+o(1)}\\
& \le N^{2+o(1)}\sum_{e \le 12N}\sum_{ b\le 12N}\sum_{a_1,a_2\le 12N/e}\min\left\{\frac{1}{b^2},\frac{1}{a^2_1a^2_2e^2} \right\} +  N^{2+o(1)}.
\end{align*}

Using  the bound on  the divisor function~\eqref{eq:div func} again we obtain
\begin{equation}
\begin{split} 
\label{eq:S1}
S_1 & \le N^{2+o(1)}\sum_{ b\le 12N} \sum_{a\le 12^4 N^2}\min\left\{\frac{1}{b^2},\frac{1}{a^2} \right\}+  N^{2+o(1)}\\
& \le N^{2+o(1)}\(\sum_{ b\le 12N} \sum_{\substack{ a\le b }}\frac{1}{b^2}+
\sum_{a\le 12^4 N^2}\sum_{\substack{b\le a }}\frac{1}{a^2}\) +  N^{2+o(1)}\\
& \le N^{2+o(1)}. 
\end{split}
\end{equation}

Finally consider $S_2$. If $d$ satisfies $\lambda_2(d)\le 1$ then by Lemmas~\ref{lem:mst} and~\ref{lem:latticesm}
\begin{equation}
\label{eq:LC1}
\# \(\cL(d)\cap B\)\ll \frac{N^3}{q}.
\end{equation}
In particular, we see that for $N = o(q^{1/3})$ the bound~\eqref{eq:LC1} implies
$$
1 \le \# \(\cL(d)\cap B\) = o(1),
$$
which means that this case (that is,  $\lambda_2(d)\le 1$) never occurs for ``small'' $N$.

For each $|n|\le 6N$ there exists at most one value of $m$ satisfying~\eqref{eq:J1def} and for any two pairs $(n_1,m_1), (n_2,m_2)$ satisfying~\eqref{eq:J1def} we have 
$$
n_1^2-n_2^2\equiv jd^2(m_1-m_2) \bmod q.
$$
This implies
$$
J(d)^2\ll \#\{ |n_1|,|n_2|, |m|\le 12N, \ n_1\neq \pm n_2:~n_1^2-n_2^2\equiv jd^2m \bmod{q}\}.
$$
Since for any  integer $r\neq 0$ the bound~\eqref{eq:div func} on the divisor function implies
$$
\# \{ |n_1|,|n_2|\le 8N :~ n_1^2-n_2^2=r\}\le N^{o(1)},
$$
we obtain
$$
J(d)^2\le \# \(\cL(d)\cap B\) N^{o(1)}.
$$
By~\eqref{eq:LC1}
$$
J(d)\ll \frac{N^{3/2+o(1)}}{q^{1/2}},
$$
which implies 
\begin{equation}
\begin{split} 
\label{eq:S2}
S_2=\sum_{\substack{d\in \F_q \\ \lambda_1(d),\lambda_2(d)\le 1}}J(d)^2\ll \frac{N^{3/2}}{q^{1/2}}\sum_{\substack{d\in \F_q \\ \lambda_1(d),\lambda_2(d)\le 1}}J(d)\ll \frac{N^{7/2+o(1)}}{q^{1/2}}.
\end{split}
\end{equation}
Combining~\eqref{eq:S0}, \eqref{eq:S1} and~\eqref{eq:S2} with~\eqref{eq:S0123}, we derive the 
desired bound on $\sT_{2,2}(N;\overline{j},q))$. 

\section{Proof of Theorem~\ref{thm:T42-N}}

\subsection{Lattices}

For a lattice $\Gamma$ and  a convex body $B$  we define the dual  lattice $\Gamma^*$ and dual body $B^*$ by
$$
\Gamma^*=\{ x\in \R^{d} :~\langle x,y \rangle \in \Z \quad \text{for all} \quad y\in \Gamma\},
$$
and 
$$
B^{*}=\{ x\in \R^{d} :~\langle x,y \rangle \le 1 \quad \text{for all} \quad y\in B\},
$$
respectively. 

The following is known as a transference theorem and is due to Mahler~\cite{Mah} which we present in 
a  form given by Cassels~\cite[Chapter~VIII, Theorem~VI]{Cass}. 
\begin{lemma}
\label{lem:transfer}
Let $\Gamma\subseteq \R^{d}$ be a lattice, $B\subseteq \R^{d}$ a symmetric convex body and let $\Gamma^{*}$ and $B^{*}$ denote the dual lattice and dual body. Let $\lambda_1,\ldots,\lambda_d$ denote the  successive minima of $\Gamma$ with respect to $B$ and $\lambda_1^{*},\ldots,\lambda_d^{*}$ the successive minima of $\Gamma^{*}$ with respect to $B^{*}$. For each $1\le j \le d$ we have
$$ \lambda_j \lambda^*_{d-j+1} \le d!.$$
\end{lemma} 

We apply Lemma~\ref{lem:transfer} to lattices of a specific type whose dual may be easily calculated. For a proof of the following, see~\cite[Lemma~15]{BK}.

\begin{lemma}
\label{lem:dualcalc}
Let $a_1,\ldots,a_d$ and $q\ge 1$ be integers satisfying $\gcd(a_i,q)=1$ and let $\cL$ denote the lattice 
$$\cL=\{ (n_1,\ldots,n_d)\in \Z^{d}:~ a_1n_1+\ldots+a_dn_d\equiv 0 \bmod{q}\}.$$
Then we have 
\begin{align*}
\cL^{*}=\biggl\{ \left(\frac{m_1}{q},\ldots,\frac{m_d}{q}\right)&  \in \Z^{d}/q  : \\
&~ \exists \  \lambda \in \Z \ \ \text{such that} \ \  a_j\lambda \equiv m_j \bmod{q}  \biggr\}.
\end{align*}
\end{lemma}

Our next result should be compared with the case $\nu=3$ of~\cite[Lemma~17]{BGKS}. It is possible to give a more direct variant of~\cite[Lemma~17]{BGKS} to estimate higher order energies of modular square roots (see the proof of Corollary~\ref{cor:4vbD} below) although this seems to put tighter restrictions on the size of the parameter $N$.

\begin{lemma}
\label{lem:shortpoint3}
Let $q$ be prime, $a,b,c\not\equiv 0 \bmod{q}$ and $L,M,N$ integers. Let $\cL$ denote the lattice 
$$
\cL=\{(\ell ,m, n) \in \Z^3   :~ a\ell +bm+c n\equiv 0 \bmod{q}\},
$$
and let $B$ be the convex body
$$
B=\{ (x,y,z)\in \R^3   :~ |x|\le  L, \ \ |y|\le M, \ \ |z|\le N \}.
$$
Let 
$$
K = \# \(\cL\cap B\),
$$
and $\lambda_1,\lambda_2$ denote the first and second successive minima of $\cL$ with respect to $B$. Then 
at least one of the following holds: 
\begin{itemize}
\item[(i)] 
$$
K <\max\left\{\frac{640LMN}{q},\,1\right\}.
$$
\item[(ii)] $\lambda_1\le 1$ and $\lambda_2>1$.  

\item[(iii)] There exists some $\lambda \not \equiv 0\bmod{q}$ and $\ell,m,n\in \Z$ satisfying 
$$
|\ell|\le \frac{4320MN}{K}, \quad  |m| \le \frac{4320LN}{K}, \quad  |n|\le \frac{4320LM}{K}
$$
and
$$
a\lambda\equiv \ell  \bmod{q}, \quad b\lambda\equiv m \bmod{q}, \quad c\lambda \equiv n \bmod{q}.
$$
\end{itemize}
\end{lemma}

\begin{proof}
Assume that~(i) fails. Thus we have
\begin{equation}
\label{eq:LuB1}
K \ge \max\left\{ \frac{640 LMN}{q},\,1\right\}.
\end{equation}
Then $K \ge 1$. Hence, if  $\lambda_1\le  \lambda_2\le \lambda_3$ denote the successive minima of $\cL$ with respect to $B$,  then  $\lambda_1\le 1$. We first show~\eqref{eq:LuB1} 
implies
$$
\lambda_3>1.
$$
Indeed, otherwise by Lemma~\ref{lem:latticesm}
\begin{equation}
\label{eq:LuB2}
K \le \(\frac{2}{\lambda_1} + 1\) \(\frac{4}{\lambda_2} + 1\)   \(\frac{6}{\lambda_3} + 1\)
 \le \frac{3}{\lambda_1}   \frac{5}{\lambda_2}  \frac{7}{\lambda_3}  =  \frac{105}{\lambda_1 \lambda_2 \lambda_3} .
\end{equation}
Since 
$$\vol(\R^3/\cL)=q \mand \vol(B) = 8LMN,
$$ 
we see from Lemma~\ref{lem:mst} that 
\begin{equation}
\label{eq:lambda prod}
 \frac{1}{\lambda_1 \lambda_2 \lambda_3} \le  \frac{3!}{8}  \frac{8LMM}{q}=  \frac{6LMN}{q},
\end{equation}
 which together with~\eqref{eq:LuB2} contradicts~\eqref{eq:LuB1}. 

 Hence we have either 
\begin{equation}
\label{eq:L3case1}
\lambda_1\le 1, \qquad \lambda_2, \lambda_3>1,
\end{equation}
or
\begin{equation}
\label{eq:L3case2}
\lambda_1, \lambda_2\le 1, \qquad \lambda_3>1.
\end{equation}

Clearly~\eqref{eq:L3case1} is the same as~(ii).

Next suppose that we have~\eqref{eq:L3case2}. By  Lemma~\ref{lem:latticesm}, a similar calculation as before, together 
with~\eqref{eq:lambda prod} gives,
\begin{equation}
\label{eq:Kcase2}
K\le   \frac{7\times15 }{\lambda_1 \lambda_2}=\frac{105\lambda_3}{\lambda_1\lambda_2\lambda_3}.
\end{equation}
Applying Lemma~\ref{lem:mst} and using 
$$
\vol (B)=8NML, \quad \vol(\R^{3}/\cL)=q,
$$
we derive from~\eqref{eq:Kcase2} that
$$
K\le \frac{105 \cdot 3! \, \vol (B) \lambda_3}{2^3\, \vol(\R^{3}/\cL)}=  \frac{630 NML \lambda_3 }{q}.
$$ 
Let $\lambda_1^{*}$ denote the first successive minima of the dual lattice $\cL^{*}$ with respect to the dual body $B^{*}$. By Lemma~\ref{lem:transfer}
$$
\lambda_3\le \frac{6}{\lambda_1^{*}}.
$$
The above estimates combined with~\eqref{eq:Kcase2} implies
$$
\lambda_1^{*} \le \frac{4320 NML}{qK}.
$$
Hence, by  the definition of $\lambda_1^{*}$
\begin{equation}
\label{eq:LstarB}
\cL^{*}\cap \frac{4320NML}{qK}B^{*}\neq \{(0,0,0)\}.
\end{equation}
Its remains to recall that by Lemma~\ref{lem:dualcalc}
\begin{align*}
  \cL^{*}& =\biggl\{ \left(\frac{\ell}{q},\frac{m}{q},\frac{n}{q}\right) \in \Z^{3}/q :~ \exists \  \lambda \in \Z  \text{ such that }\\
&\qquad \qquad  \qquad  a\lambda \equiv \ell \bmod{q}, \ b\lambda\equiv m \bmod{q}, \ c\lambda \equiv n \bmod{q}\biggr  \},
\end{align*}
and  also it is obvious that 
$$
B^{*}=\{ (x,y,z)\in \R^3  :~ L|x|+M|y|+ N|z|\le 1\}.
$$
By~\eqref{eq:LstarB}, this implies there exists some $\lambda \not \equiv 0\bmod{q}$ and $\ell ,m, n$ satisfying~(iii), 
which completes the proof.
\end{proof}

\begin{corollary}
\label{cor:4vbD} Let $\varepsilon>0$ be a fixed real number.
For $j\in \Fq^{*}$, integer $N\ll q$ and $\Delta\ge 1$, let $\cA,\cD\subseteq \Fq$ denote the sets
$$
\cA=\{ x\in \F_q^{*}   :~ jx^2 \in [1,N] \}.
$$
and 
$$
\cD=\{ d\in \Fq^{*}   :~ (\cA\circ \cA)(d)\ge \Delta\}.
$$
Let $K$ be sufficiently large and suppose $K$ and $\Delta$ satisfy 
\begin{equation}
\label{eq:D3Kcond}
K\ge  \(\frac{N^{15/2}}{\Delta^{12}q^{1/2}}+\frac{N^{5}}{\Delta^{8}q^{1/4}}\) N^{\varepsilon}
\end{equation} 
and 
\begin{equation}
\label{eq:Deltaassump}
\Delta \ge \(\frac{N^{3/2}}{q^{1/2}}+\frac{N^{5/8}}{q^{1/8}}\) N^{\varepsilon}.
\end{equation}
Let $\cF\subseteq \F_q^{*}$ denote the set of $f$ satisfying 
\begin{equation}
\label{eq:DKb}
(\cD\circ \cD)(f)\ge K.
\end{equation}
Then  either
\begin{equation}
\label{eq:Kcond1-1}
K\ll 1,
\end{equation}
or 
$$
K\#\cF\ll \frac{N^{3+o(1)}}{\Delta^4}.
$$
\end{corollary}

\begin{proof}
From~\eqref{eq:DKb} 
\begin{equation}
\label{eq:DDK1}
K\le \# \{ (d_1,d_2)\in \cD^2 :~d_1-d_2=f\}.
\end{equation}
If $d_1,d_2\in \cD$ satisfy $d_1-d_2=f$, then 
$$
d_1^2-d_2^2 - f^2 = (d_1-d_2)^2 + 2d_1d_2 - 2d_2^2 - f^2  
= 2d_2 (d_1-d_2) = 2d_2 f
$$ 
and 
some algebraic manipulations show
$$
(2jd_1^2-2jd_2^2-2jf^2)^2=8j f^2(2j d_2^2).
$$
Since $0\not \in \cD$, for each $d\in \cD$, by~\eqref{eq:Deltaassump} and~\cite[Lemma~6.4]{DKSZ} there exists $m_d,n_d$ satisfying 
\begin{equation}
\label{eq:Lemma64DKSZ}
\begin{split} 
&2jd^2\equiv  m_d^{-1}n_d \bmod{q}, \qquad |n_d|\ll \frac{N^2}{\Delta^2}, \\
&  |m_d|\ll \frac{N}{\Delta^2}, \qquad \gcd(m_d,n_d)=1.
\end{split} 
\end{equation}
Let $I(f)$ count the number of solutions to the congruence 
\begin{equation}
\label{eq:reDnm}
\left(n_{d_1}m_{d_1}^{-1}-n_{d_2}m_{d_2}^{-1}-2jf^2\right)^2\equiv 8jf^2n_{d_2}m_{d_2}^{-1} \bmod{q},
\end{equation}
with $d_1,d_2\in \cD$.
The above and~\eqref{eq:DDK1} imply
\begin{equation}
\label{eq:KIf}
K\le I(f).
\end{equation}
Rearranging~\eqref{eq:reDnm} we obtain 
$$
\left(m_{d_2}n_{d_1}-m_{d_1}n_{d_2} -2jf^2m_{d_1}m_{d_2}\right)^2\equiv 8jf^2m_{d_1}^2m_{d_2}n_{d_2} \bmod{q}.
$$
This implies that $I(f)$ is bounded by the number of solutions to
\begin{equation}
\begin{split} 
\label{eq:If123}
(n_{d_1}m_{d_2}-n_{d_2}m_{d_1})^2 & -4jf^2m_{d_1}m_{d_2}(n_{d_1}m_{d_2}+n_{d_2}m_{d_1})\\
& \qquad \qquad  +4j^2f^4(m_{d_1}m_{d_2})^2\equiv 0 \bmod{q},
\end{split} 
\end{equation}   
with $d_1,d_2\in \cD$. Let $\cL$ denote the lattice 
$$
\cL=\{ (m,n,\ell)\in \Z^{3}:~ m+njf^2+\ell j^2f^4\equiv 0 \bmod{q}\},
$$
and $B$ the convex body 
$$
B=\left\{ (x,y,z)\in \R^3:~ |x|\le \frac{CN^6}{\Delta^8}, \  |y|\le \frac{CN^5}{\Delta^8}, \  |z|\le\frac{CN^4}{\Delta^8}\right\}.
$$  
for a suitable absolute constant $C$. By~\eqref{eq:Lemma64DKSZ} and~\eqref{eq:If123}
\begin{equation}
\begin{split}
\label{eq:d123456789}
\bigl((n_{d_1}m_{d_2}-n_{d_2}m_{d_1})^2, -4m_{d_1}m_{d_2}(n_{d_1}m_{d_2}+n_{d_2}m_{d_1}),&\\
  4(m_{d_1}m_{d_2})^2\bigr)&\in \cL\cap B.
\end{split} 
\end{equation}
Let $\lambda_1,\lambda_2$ denote the first and second successive minima of $\cL$ with respect to $B$. Assuming that $K\ge 1$ we have $\lambda_1\le 1$. 

Suppose that 
$$
\lambda_1\le 1, \quad \lambda_2>1.
$$
Then there exists some $(a_0,b_0,c_0)\in \cL\cap B$ such that for any $d_1,d_2\in \cD$ satisfying~\eqref{eq:d123456789} we have 
\begin{align*}
\left((n_{d_1}m_{d_2}-n_{d_2}m_{d_1})^2, -4m_{d_1}m_{d_2}(n_{d_1}m_{d_2}+n_{d_2}m_{d_1}),4(m_{d_1}m_{d_2})^2\right)&
\\ =m(a_0,b_0,&c_0),
\end{align*}
for some $m\in \Z$. Note from~\eqref{eq:Lemma64DKSZ} for each $d_1,d_2\in \cD$ we have  $m_{d_1}m_{d_2}\neq 0$ and hence $c_0\neq 0$. This implies  
\begin{align*}
& \left(\frac{n_{d_1}}{m_{d_1}}-\frac{n_{d_2}}{m_{d_2}}\right)^2=\frac{a_0}{c_0},\\
& \frac{n_{d_1}}{m_{d_1}}+\frac{n_{d_2}}{m_{d_2}}=\frac{b_0}{c_0}.
\end{align*}
Hence
\begin{align*}
K\le \#\biggl\{ (d_1,d_2)\in \cD\times \cD : ~ \frac{n_{d_1}}{m_{d_1}}-\frac{n_{d_2}}{m_{d_2}}&=\pm\left(\frac{a_0}{c_0}\right)^{1/2}, \\
&   \frac{n_{d_1}}{m_{d_1}}+\frac{n_{d_2}}{m_{d_2}}=\frac{b_0}{c_0}  \biggr\} \le 8,
\end{align*}
since once $n_{d_1}/m_{d_1}$ is fixed, due to the coprimality condition 
in~\eqref{eq:Lemma64DKSZ}, $d^2_1$ is uniquely defined, and similarly for $d^2_2$. 
This implies~\eqref{eq:Kcond1-1}.  

Suppose next that 
\begin{equation}
\label{eq:l1l2-case2}
\lambda_1\le 1, \quad \lambda_2\le 1.
\end{equation}

Let $J(\ell,m,n)$ count the number of solutions to 
$$
m_{1}m_{2}=\ell, \quad n_{1}m_{2}+n_{2}m_{1}=m, \quad n_{1}m_{2}-n_{2}m_{1}=n,
$$
with 
\begin{equation}
\label{eq:n12m12}
|m_1|, |m_2|\ll \frac{N}{\Delta^2}, \quad |n_1|,|n_2|\ll \frac{N^2}{\Delta^2}, \quad   \quad m_1m_2n_1n_2\neq 0,
\end{equation}
 so that
\begin{equation}
\label{eq:IFzz66}
I(f)\ll \sum_{\substack{|m|,|n|\le CN^3/\Delta^4 \\ |\ell|\le CN^2/\Delta^4 \\ 4j^2f^4\ell^2 -4jf^2\ell m+ n^2 \equiv 0 \bmod{q}}}J(\ell,m,n),
\end{equation}
 for some absolute constant $C$. We next show that 
\begin{equation}
\label{eq:Jbnml}
J(\ell,m,n)=N^{o(1)}.
\end{equation} 
Estimates for the divisor function~\eqref{eq:div func} imply the number of solutions to 
$$
m_1m_2=\ell, \quad \text{$m_1,m_2$ satisfying~\eqref{eq:n12m12}},
$$
is at most $N^{o(1)}$. For each such $m_1,m_2$ there exists at most one solution to the system 
$$
n_1m_2-n_2m_1=n, \quad n_1m_2+n_2m_1=m, \quad \text{$n_1,n_2$ satisfying~\eqref{eq:n12m12}},
$$
which establishes~\eqref{eq:Jbnml}. By~\eqref{eq:KIf} and~\eqref{eq:IFzz66}
\begin{align*}
K \le \# \{(\ell,m,n) \in \Z^3&:~ |\ell|\le CN^2/\Delta^4 , \  |m|,|n| \le CN^3/\Delta^4, \ \\  & \quad n^2-4jf^2\ell m+4j^2f^4\ell^2\equiv 0 \bmod{q} \} N^{o(1)},
\end{align*}  
and hence 
\begin{equation}
\begin{split} 
\label{eq:KUBL}
K\le \# \Bigl\{(\ell,m,n) \in \Z^3&:\\
|\ell|   \le 2CN^2/\Delta^4,    &\ |m|\le  4C^2N^5/\Delta^8,   \ |n| \le CN^3/\Delta^4,   \\   &n^2+jf^2m+j^2f^4\ell^2\equiv 0 \bmod{q}
\Bigr \} N^{o(1)}. 
\end{split}
\end{equation}
By~\eqref{eq:Deltaassump}, for each $\ell,n \in \Z$, there exists at most one value of $|m|\ll N^5/\Delta^8$ satisfying 
$$
n^2+jf^2m+j^2f^4\ell^2\equiv 0 \bmod{q}.
$$
For any $(\ell_1,m_1,n_1)$ and $(\ell_2,m_2,n_2)$ satisfying the conditions of~\eqref{eq:KUBL}, there exists some $|m|\ll N^5/\Delta^8$ such that 
\begin{equation}
\label{eq:2points}
n_1^2+n_2^2-2jf^2m+j^2f^4(\ell^2_1+\ell_2^2)\equiv 0 \bmod{q}.
\end{equation}
Define the lattice 
$$
\cL=\{ (n,m,\ell)\in \Z^3 : ~n+jf^2m +j^2f^4\ell \equiv 0 \bmod{q} \},
$$
and the convex body 
\begin{align*}
B=\{ (n,m,\ell)\in \R^3   :~ |n|&\le C_0N^6/\Delta^8, \\
&  |m|\le  C_0N^5/\Delta^8,   \  |\ell|\le C_0N^4/\Delta^8  \},
\end{align*}
for a suitable constant $C_0$. Since for any integer $r$
$$
\# \{ n_1,n_2\in \Z   :~ n_1^2+n_2^2=r\} \le r^{o(1)},
$$
we see that~\eqref{eq:2points} implies 
$$
K^2\le \# \(\cL\cap B\)  N^{o(1)}.
$$
By~\eqref{eq:D3Kcond}, \eqref{eq:l1l2-case2} and Lemma~\ref{lem:shortpoint3}, there exists $(\ell,m,n)\neq (0,0,0)$ satisfying  
\begin{equation}
\label{eq:nml-1-1-1}
 |\ell|\le \frac{N^{11+o(1)}}{\Delta^{16}K^2}, \qquad |m|\le \frac{N^{10+o(1)}}{\Delta^{16}K^2}, \qquad 
 |n|\le \frac{N^{9+o(1)}}{\Delta^{16}K^2},
\end{equation} 
and 
\begin{equation}
\label{eq:jjf}
jf^2n\equiv m \bmod{q}, \quad j^2f^4n\equiv \ell \bmod{q}.
\end{equation}
Note we may assume
\begin{equation}
\label{eq:jfnmlcond}
\gcd(\ell, m,n)=1.
\end{equation}
Recall~\eqref{eq:If123}
\begin{equation}
\begin{split}
\label{eq:If12345}
 I(f)\le\#\{ (d_1,d_2) & \in \cD^2 :~   (n_{d_1}m_{d_2}-n_{d_2}m_{d_1})^2\\
&      -4jf^2m_{d_1}m_{d_2}(n_{d_1}m_{d_2}+n_{d_2}m_{d_1})  \\
& \qquad \qquad +4j^2f^4(m_{d_1}m_{d_2})^2\equiv 0 \bmod{q}\}.
\end{split} 
\end{equation}
If $d_1,d_2$ satisfy the conditions in~\eqref{eq:If12345}, then by~\eqref{eq:jjf}
\begin{align*}
& n(n_{d_1}m_{d_2}-n_{d_2}m_{d_1})^2-4mm_{d_1}m_{d_2}(n_{d_1}m_{d_2}+n_{d_2}m_{d_1})   \\ & \quad \qquad \qquad \qquad \qquad \qquad \qquad +4\ell (m_{d_1}m_{d_2})^2\equiv 0 \bmod{q},
\end{align*}
and hence from~\eqref{eq:D3Kcond}, assuming that $N$ is large enough, we derive
\begin{equation}
\begin{split}
\label{eq:limini}
  n(n_{d_1}m_{d_2}-n_{d_2}m_{d_1})^2-4mm_{d_1}m_{d_2}(n_{d_1}&m_{d_2}+n_{d_2}m_{d_1})   \\  &+4\ell (m_{d_1}m_{d_2})^2=0.
\end{split} 
\end{equation}
Similarly by~\eqref{eq:nml-1-1-1} and~\eqref{eq:jjf} we have $m^2\equiv n\ell  \bmod{q}$ and again~\eqref{eq:D3Kcond} ensures that 
$$
m^2=n\ell.
$$
Therefore~\eqref{eq:limini} implies the following equation 
$$
 \left(\frac{n_{d_1}}{m_{d_1}}-\frac{n_{d_2}}{m_{d_2}}\right)^2-4\left(\frac{n_{d_1}}{m_{d_1}}+\frac{n_{d_2}}{m_{d_2}}\right)\left(\frac{m}{n}\right)+4\left(\frac{m}{n}\right)^{2}=0.
$$
We see that 
\begin{equation}
\label{eq:nm123123123}
\frac{m}{n}=\frac{1}{2}\left(\frac{n_{d_1}}{m_{d_1}}+\frac{n_{d_2}}{m_{d_2}}\right)\pm \frac{\sqrt{n_{d_1}m_{d_1}n_{d_2}m_{d_2}}}{m_{d_1}m_{d_2}}.
\end{equation}
Hence from~\eqref{eq:Lemma64DKSZ} and~\eqref{eq:If12345}, there exists some constant $C$ such that
\begin{align*}
I(f)&\le \#\biggl\{ \(m_{d_1},m_{d_2},n_{d_1},n_{d_2} \)\in \Z^4:\\
& \qquad \qquad  |m_{d_1}|,|m_{d_2}|\le \frac{CN}{\Delta^2}, \  |n_{d_1}|,|n_{d_2}|\le \frac{CN^2}{\Delta^2}, 
\\ & \qquad \qquad \qquad\qquad m_{d_1}m_{d_2}n_{d_1}n_{d_2}\neq 0, 
\ \text{and~\eqref{eq:nm123123123} holds}   \biggr\}.
\end{align*}
Summing the above over $f\in \cF$, using~\eqref{eq:KIf} and noting that for each $\ell ,m, n$ satisfying~\eqref{eq:jfnmlcond} there exists $O(1)$ values of $f$ satisfying~\eqref{eq:jjf}, we see that $K\# \cF$ is bounded by the number of solutions to the equation~\eqref{eq:nm123123123} with integer variables satisfying 
$$
|m_{d_1}|,|m_{d_2}|\le \frac{CN}{\Delta^2}, \qquad |n_{d_1}|,|n_{d_2}|\le \frac{CN^2}{\Delta^2}, \qquad   n_{d_1}n_{d_2}m_{d_1}m_{d_2}\neq 0. 
$$
We see from~\eqref{eq:nm123123123} 
that $n_{d_1}m_{d_1}n_{d_2}m_{d_2}=r^2$ for some $r\in \Z$ and hence a  bound~\eqref{eq:div func} on the divisor function, implies 
$$
K\#\cF\le  N^{o(1)}\#\left\{ \ell \le C^4 \frac{N^6}{\Delta^8} :~\ell=r^2 \ \text{for some $r \in \Z$}\right\}\le \frac{N^{3+o(1)}}{\Delta^4},
$$
which completes the proof.
\end{proof}

\subsection{Concluding the proof} As in Section~\ref{sec:RedLatt}, 
here we  work  again with $\sT_{4,2}(N;\overline{j},q)$ for  the multiplicative inverse 
$\overline{j}$ of $j$ modulo $q$ rather than with  $\sT_{4,2}(N;j,q)$.
Let notation be as in Corollary~\ref{cor:4vbD}, so that
$$
\sT_{4,2}(N;\overline{j},q)=\sum_{x\in \F_q}(\cA\circ \cA \circ \cA \circ \cA)(x)^{2},
$$  
where we recall that
$$\cA=\{x\in \F^{*}_q \ : \ jx^2\in [1,N]\}.$$
By~\eqref{eq:T42123} we may assume that
\begin{equation}
\label{eq:Nassumption12}
N\le q^{1/3}.
\end{equation}  
 Applying the dyadic pigeonhole principle, there exist $\Delta_1,\Delta_2 \ge 1$ and $\cD_1,\cD_2\subseteq \Fq$ given by 
$$
\cD_j=\{ x\in \Fq   :~ \Delta_j \le (\cA\circ \cA)(x)< 2\Delta_j\}, \qquad j=1,2, 
$$
such that
$$\sT_{4,2}(N;\overline{j},q) \le N^{o(1)}(\Delta_1\Delta_2)^2E(\cD_1,\cD_2),$$
where  
$$
E(\cD_1,\cD_2)=\sum_{x\in \F_q}(\cD_1\circ \cD_2)(x)^{2}.
$$
By the Cauchy-Schwarz inequality 
$$
E(\cD_1,\cD_2)\le E(\cD_1)^{1/2}E(\cD_2)^{1/2},
$$
and hence there exists some $\Delta$ and $\cD$ given by 
$$
\cD=\{ x\in \Fq   :~ \Delta \le (\cA \circ \cA)(x)< 2\Delta\},
$$
such that 
\begin{equation}
\label{eq:T4ED}
\sT_{4,2}(N;\overline{j},q) \le N^{o(1)}\Delta^{4}E(\cD).
\end{equation}
It is also obvious  from~\eqref{eq:EAI} that 
\begin{equation}
\label{eq:T2D2}
\Delta^2\(\# \cD \)\le \sT_{2,2}(N;\overline{j},q),
\end{equation}
and 
\begin{equation}
\label{eq:Dub-12}
\# \cD \le \Delta \#\cD \ll N^2.
\end{equation}

Isolating the diagonal contribution in $E(\cD)$, we write
$$
E(\cD)=(\#\cD)^2+\sum_{f\in \Fq^{*}}(\cD \circ \cD)(f)^2.
$$
We may assume 
\begin{equation}
\label{eq:EDED}
E(\cD)\le 2 \sum_{f\in \Fq^{*}}(\cD \circ \cD)(f)^2,
\end{equation}
since otherwise we have $E(\cD) \le 2 (\#\cD)^2$ and it  follows from the bounds~\eqref{eq:T4ED}  and~\eqref{eq:T2D2}  that 
 $$
\sT_{4,2}(N;\overline{j},q)  \le  \Delta^{4}(\#\cD)^2 N^{o(1)} \le \sT_{2,2}(N;\overline{j},q)^2 N^{o(1)} .
$$
Now, recalling the condition~\eqref{eq:Nassumption12} and using
Theorem~\ref{thm:T22-N}, we derive 
$$
\sT_{4,2}(N;\overline{j},q) \le N^{4+o(1)}.
$$

By~\eqref{eq:EDED} and the dyadic pigeonhole principle there exists some $K$ and a set 
$\cF\subseteq \Fq^{*}$ given by 
$$
\cF=\{ f\in \Fq^{*}   :~ K\le (\cD\circ \cD)(f)<2K\},
$$
such that
\begin{equation}
\label{eq:fDfD}
E(\cD) \le  K^{2}\#\cF N^{o(1)}.
\end{equation} 
Combining with~\eqref{eq:T4ED} and~\eqref{eq:fDfD} gives
\begin{equation}
\label{eq:T4dyadic}
\sT_{4,2}(N;\overline{j},q) \le  \Delta^{4}K^{2}\#\cF N^{o(1)}.
\end{equation}
We  apply Corollary~\ref{cor:4vbD} to estimate the right hand side of~\eqref{eq:T4dyadic}.

We now fix some $\varepsilon > 0$ and suppose first that one of~\eqref{eq:D3Kcond} or~\eqref{eq:Deltaassump} does not hold. In particular, assume
\begin{equation}
\label{eq:D3Kcond-1}
K < \left(\frac{N^{15/2}}{\Delta^{12}q^{1/2}}+\frac{N^{5}}{\Delta^{8}q^{1/4}}\right) N^{\varepsilon}
\end{equation}
or
\begin{equation}
\label{eq:Deltaassump-1}
\Delta <  \frac{N^{5/8+\varepsilon}}{q^{1/8}},
\end{equation}
where we have use the assumption~\eqref{eq:Nassumption12} to ignore the term  $N^{3/2}/q^{1/2}$ in~\eqref{eq:Deltaassump}.
If~\eqref{eq:D3Kcond-1} holds, then using the trivial bounds 
 $$
 K \#\cF \le (\#\cD)^2 \mand  \Delta  \#\cD \ll N^2\,,
 $$ 
we derive from~\eqref{eq:T4dyadic}
  \begin{equation}\label{eq:T42-1}
  \begin{split}
 \sT_{4,2}(N;\overline{j},q)& \le \Delta^{4}(\#\cD)^2K N^{o(1)} \le \Delta^2 K N^{4+o(1)}  \\
 &\le\left(\frac{N^{15/2}}{\Delta^{10}q^{1/2}}+\frac{N^{5}}{\Delta^{6}q^{1/4}}\right)
N^{4+\varepsilon+o(1)} \\
&\le \left(\frac{N^{15/2 }}{q^{1/2}}+\frac{N^{5}}{q^{1/4}}\right)N^{4+\varepsilon+o(1)}\\
& \le \left(\frac{N^{11/2 }}{q^{1/2}}+\frac{N^{3}}{q^{1/4}}\right)N^{6+\varepsilon+o(1)}
\end{split}
\end{equation}
If~\eqref{eq:Deltaassump-1} holds, then from~\eqref{eq:T4dyadic}
  \begin{equation}\label{eq:T42-2}
  \begin{split}
\sT_{4,2}(N;\overline{j},q)& \le N^{o(1)}\Delta^{4}(\# \cD)^3 \le N^{6+o(1)}\Delta\\
& \le \frac{N^{6+5/8+o(1)}}{q^{1/8}}.
\end{split}
\end{equation}
Hence if one of the conditions~\eqref{eq:D3Kcond} or~\eqref{eq:Deltaassump} does not hold then 
combining~\eqref{eq:T42-1} and~\eqref{eq:T42-2} we obtain
  \begin{equation}\label{eq:T42-Comb}
\sT_{4,2}(N;\overline{j},q) \le \left(\frac{N^{5/8}}{q^{1/8}}+\frac{N^{8}}{q^{1/2}}\right)N^{6+\varepsilon+o(1)}.
\end{equation}

Suppose next that~\eqref{eq:D3Kcond-1} and~\eqref{eq:Deltaassump-1} both fail and thus
both~\eqref{eq:D3Kcond} and~\eqref{eq:Deltaassump} hold. By Corollary~\ref{cor:4vbD} we have either 
\begin{equation}
\label{eq:case1-12}
K\ll 1,
\end{equation}
or 
\begin{equation}
\label{eq:case2-12}
K \#\cF  \le \frac{N^{3+o(1)}}{\Delta^4}.
\end{equation}

If~\eqref{eq:case1-12} holds then from~\eqref{eq:T4dyadic} and the trivial bound  $K \#\cF \le (\#\cD)^2$, 
we derive
$$
\sT_{4,2}(N;\overline{j},q)\le  \Delta^{4}K^2\#\cF N^{o(1)} \le \Delta^{4}K\#\cF N^{o(1)} \le \Delta^4(\#\cD)^2 N^{o(1)}.
$$
Now the bound~\eqref{eq:T2D2} and Theorem~\ref{thm:T22-N}
(under the condition~\eqref{eq:Nassumption12}), yield
$$
\sT_{4,2}(N;\overline{j},q)   \le  T_{2,2}(N;j,q)^2 N^{o(1)}\le N^{4+o(1)}.
$$

If~\eqref{eq:case2-12} holds then using~\eqref{eq:Dub-12}
  \begin{equation}\label{eq:T42-3}
\sT_{4,2}(N;\overline{j},q)\le N^{3+o(1)}K \le N^{3+o(1)}\#\cD \le N^{5+o(1)}.
\end{equation}

Combining~\eqref{eq:T42-Comb} and~\eqref{eq:T42-3}, since $\varepsilon>0$ is arbitrary, we complete the proof.

\section{Proof of Theorem~\ref{thm:E2k}}
  
\subsection{Product polynomials }
\label{sec:prod poly}
In the proof of~\cite[Lemma~5.1]{SSZ1}, a certain polynomial in four variables with integer
coefficients played a key role. More precisely, it has been found in~\cite{SSZ1} that the polynomial
\begin{align*} 
F(U,V,X,Y) & = 64 UVXY\\
& \qquad  -  \(4U V + 4XY - \(X +Y -U - V\)^2\)^2 \,,
\end{align*}
has the following property. Letting $U = u^2$, $V = v^2$, $X = x^2$, and $Y = y^2$,
one has that $F(u^2, v^2, x^2, y^2) = 0$ for any $u, v, x, y$ for which $u+v = x+y$ (over any commutative ring). We now
proceed to discuss this property in a more general context.  

Denote 
$\cU_k = \{ \omega \in \C :~\omega^k = 1\}$ 
and  consider the polynomial
$$
G_k(X_1,X_2, X_3, X_4) = \prod_{\omega_1,\omega_2,\omega_3  \in\cU_k}
(\omega_1 X_1+ \omega_2 X_2 -  \omega_{3} X_{3} - X_4)
$$
defined over the cyclotomic field $K_k = \Q\(\exp(2 \pi i /k)\)$. 
Since the Galois group  $\Gal(K_k/\Q)$ of $K$ is cyclic and any automorphism $\sigma$ of $K_k$ over $\Q$ 
is a  multiplication by 
some $\omega \in \cU_k$, we see that
\begin{align*}
\sigma&\(G_k(X_1,X_2, X_3, X_4)\) \\
& =  \prod_{\omega_1, \omega_2, \omega_3  \in\cU_k}
\(\sigma\(\omega_1\)   X_1+ \sigma\(\omega_2\) X_2 -  \sigma\(\omega_{3}\) X_{3} - \sigma\(1\) X_4\)\\
 & =  \prod_{\omega_1, \omega_2, \omega_3  \in\cU_k}
\(\omega \omega_1   X_1+ \omega \omega_2  X_2 - \omega \omega_{3}  X_{3} -\omega X_4\)\\
&= \omega^{k^3}  \prod_{\omega_1, \omega_2, \omega_3  \in\cU_k}
\(\omega_1 X_1+ \omega_2 X_2 -  \omega_{3} X_{3} - X_4\)\\
&= G_k(X_1,X_2, X_3, X_4).
\end{align*}
Hence $G_k$ has rational coefficients. Since obviously these coefficients are algebraic integers, we see that 
$G_k\(X_1,X_2, X_3, X_4\) \in \Z[X_1,X_2, X_3, X_4]$.

We also see that 
\begin{align*}
  \prod_{\omega_1, \omega_2, \omega_3 \in\cU_k}&
\(\omega_1 X_1+ \omega_2 X_2 -  \omega_{3} X_{3} - X_4\)\\
 & =  \prod_{\omega_1, \omega_2, \omega_3  \in\cU_k}
\(\omega_1 X_1+  \omega_1  \omega_2 X_2 - \omega_1  \omega_{3} X_{3} - X_4\)\\
& =  \prod_{\omega_2, \omega_3 \in \cU_k}
  \prod_{\omega_1 \in \cU_k} 
\(\omega_1 \(X_1+    \omega_2 X_2 -  \omega_{3} X_{3}\) - X_4\)\\
& =  (-1)^k \prod_{\omega_2, \omega_3 \in \cU_k} \( \(X_1+    \omega_2 X_2 -  \omega_{3} X_{3}\)^k - X_4^k\)
\end{align*}
Therefore $G_k(X_1,X_2, X_3, X_4)$ is a polynomial in $X_4^k$. Similarly,
\begin{align*}
  \prod_{\omega_1, \omega_2, \omega_3 \in\cU_k}&
\(\omega_1 X_1+ \omega_2 X_2 -  \omega_{3} X_{3} - X_4\)\\
 & =  \prod_{ \omega_2, \omega_3  \in\cU_k}   \prod_{\omega_1 \in \cU_k} 
\(  X_1+  \omega_1^{-1}  \( \omega_2 X_2 -   \omega_{3} X_{3} - X_4\)\)\\
& =  \prod_{\omega_2, \omega_3 \in \cU_k}
\( X_1^k+    \( \omega_{3} X_{3} +X_4- \omega_2 X_2 \)^k\)
\end{align*}
Thus, it is also a polynomial in $X_1^k$ and of course also in $X_2^k$ and $X_3^k$. 
Hence we can write 
$$
G_k(X_1,X_2, X_3, X_4) = F_k\(X_1^k,X_2^k, X_3^k, X_4^k\)
$$
for some polynomial $F_k\(X_1,X_2, X_3, X_4\) \in \Z[X_1,X_2, X_3, X_4]$. 

\begin{remark}
It is clear that this construction can be extended in several directions, 
in particular to polynomials $F_{\nu,k} \in \Z[X_1, \ldots, X_{2\nu}]$ such that 
$$
F_{\nu,k}\(x_1^k, \ldots, x_{2\nu}^k\) = 0
$$
whenever $x_1+ \ldots +x_\nu = x_{\nu+1} + \ldots + x_{2\nu}$. 
\end{remark}   
 
\subsection{The zero set of $F_k(X_1,X_2,X_3,X_4)$} 
We now need the following bound on the number of integer zeros of $F_k$ in a box.
Define by $T_k(N)$ by
\begin{align*}
T_k(N)=\# \{ (n_1, n_2, n_3, n_4) \in \Z^4 :~  1 &\le n_1,n_2,n_3, n_4 \le N,\\
&  \qquad F_k(n_1,n_2,n_3,n_4) = 0 \}.
\end{align*}
Our next result gives a bound for $T_k(N)$.
\begin{lemma}\label{lem: Bound T(N)}
Fix an integer $k \ge 3$. For any positive integer $N$, we have $T_k(N) \ll  N^{2}$. 
\end{lemma}

\begin{proof}
 Take a solution $(n_1, n_2, n_3, n_4)$  to $F_k( n_1, n_2, n_3, n_4) = 0$ 
satisfying 
$1 \le n_1, n_2, n_3, n_4 \le N$. Denote by $t_1, t_2, t_3, t_4$ the positive real numbers that are
roots of order $k$ of $n_1, n_2, n_3, n_4$ respectively.   

Therefore there exist roots of unity $\omega_1,\omega_2,\omega_3 \in \mathcal{U}_k$ such that
\begin{equation}\label{eq:lin eq ti}
 \omega_1 t_1 + \omega_2 t_2 - \omega_3 t_3 - t_4  = 0.
\end{equation}

We now distinguish two cases.

{\it Case~1.} At least one of the roots of unity $\omega_1$, $\omega_2$, $\omega_3$ is not real.
Complex conjugation then provides a second linear equation,
\begin{equation}
\label{eq:conj lin eq ti}
 \bar \omega_1 t_1 + \bar \omega_2 t_2 - \bar \omega_3 t_3 - t_4  = 0.
\end{equation}
 which is different from~\eqref{eq:lin eq ti}.  
 Then using~\eqref{eq:lin eq ti} and~\eqref{eq:conj lin eq ti} to eliminate $t_4$ one obtains a nontrivial linear equation
in $t_1, t_2$ and $t_3$ which obviously has at most $O(N^2)$ solutions, after which $t_4$ is uniquely defined. 

Thus the total number of solutions in Case~1 is $O(N^2)$.

{\it Case~2.}  All three of $\omega_1, \omega_2, \omega_3$ are real, that is, $\omega_1, \omega_2, \omega_3 \in \{ -1, 1\}$, and the
equation~\eqref{eq:lin eq ti} reduces to
\begin{equation}\label{eq: real lin eq ti}
 t_1 \pm  t_2 \pm  t_3 \pm t_4  = 0.
\end{equation}
We observe that {\it Case~2\/}  also covers the $2N^2 + O(N)$ diagonal solutions. 

To treat the non-diagonal solutions, one can now apply results of
Besicovitch~\cite{B}, Mordell~\cite {M},  Siegel~\cite{S}, or the more recent results
of Carr and O'Sullivan~\cite{CS}. For instance,~\cite[Theorem~1.1]{CS} shows that
a set of real $k$-th roots of integers that are pairwise linearly independent over the
rationals must also be linearly independent. Applying this to the set $t_1, t_2, t_3, t_4$,
which by~\eqref{eq: real lin eq ti} is not linearly independent over $\Q$, it follows that
two of them,  for example, $t_1$ and $t_2$, are linearly dependent over $\Q$. We 
derive that there are positive integers $a_1, a_2, b$ such that 
$$
t_1^k = n_1 = b a_1^k \mand t_2^k = n_2 = b a_2^k .
$$ 
where $b$ is not divisible by a $k$-th power of a prime. 
That is, 
$a_1^k$ is the largest $k$-th power that divides $n_1$,
and $a_2^k$ is the largest $k$-th power that divides $n_2$.

Then letting $t_5$ denote the positive $k$-th
root of $b$, the equation~\eqref{eq: real lin eq ti} becomes
\begin{equation}\label{eq: a12 t534}
(a_1 \pm a_2) t_5 \pm t_3 \pm t_4 = 0.
\end{equation}

Without loss of generality, we can assume that $a_1 \ge a_2$. 
Hence for any fixed $1\le a_2 \le a_1 \le N^{1/k}$ there are at most $N/a_1^k$ possible 
values for $b$ and thus for $t_5$. After $a_1$, $a_2$ and $t_5$ are fixed, 
there are obviously at most $N$ pairs $(t_3,t_4)$ satisfying~\eqref{eq: a12 t534}.
Hence the total contribution from such solutions is 
$$
\sum_{ 1\le a_2 \le a_1 \le N^{1/k}} N^2/a_1^k 
\le  \sum_{ 1  \le a_1 \le N^{1/k}} N^2/a_1^{k-1}  
\ll N^2
$$
which concludes the proof. 
\end{proof}

We remark that the case of $k=2$ can also be included in Lemma~\ref{lem: Bound T(N)}
however this case is already fully covered by the results of~\cite{SSZ1}.

\subsection{Concluding the proof}  Clearly the congruence
$$
u+v \equiv x + y\bmod q, \qquad ju^k,jv^k, jx^k, jy^k \in [1,N]
$$
implies that 
$$
F_k(u^k,v^k, x^k, y^k )  \equiv  0 \bmod q
$$
for the above  polynomial $F_k$. Since $F_k$ is homogenous this implies that 
$$
F_k(ju^k,jv^k, jx^k, jy^k )  \equiv  0 \bmod q.
$$

Since for a prime $q\sim Q$, $a \in \F_q$ and $j \in \F_q^*$, there are at most $k$ solutions to the congruence 
$jz^k \equiv a \bmod q$ in variable $z\in \F_q$, and thus at most $2k$ solution in variable $z \in [1,N]$ (since $N \le Q \le 2q$)    we have 
\begin{align*}
\sum_{\substack{ q  \sim Q \\ q~\text{prime}}}  \max_{j \in \F_q^*}   \sE_k(N;j,q) & = \sum_{\substack{ q  \sim Q \\ q~\text{prime}}}  \max_{j \in \F_q^*}   \sE_k(N;\overline{j},q) \\
& \le   16k^4 \sum_{\substack{ q  \sim Q \\ q~\text{prime}}}  
\, \ssum_{\substack{U,V,X,Y \in[1,N]\\ F_k(U,V,X,Y) \equiv 0 \bmod q}} 1 . 
\end{align*}
where, as before,  $\overline{j}$  denotes  the multiplicative inverse 
of $j$ modulo $q$.
Changing  the order of summation and separating the sum over the variables  $U,V,X,Y$ into two parts depending 
on whether  $F_k(U,V,X,Y)  = 0$ or not, we derive 
\begin{align*}
\sum_{\substack{ q  \sim Q \\ q~\text{prime}}} &  \max_{j \in \F_q^*}   \sE_k(N;j,q)   \ll   \ssum_{U,V,X,Y \in [1,N]} \, \sum_{\substack{q  \sim Q \\ q~\text{prime} \\ q \mid F_k(U,V,X,Y)}}  1\\
 & \ll  \frac{Q}{\log Q}   \ssum_{\substack{U,V,X,Y \in[1,N]\\ F_k(U,V,X,Y) = 0}} 1 +   \ssum_{\substack{U,V,X,Y \in [1,N]\\ F_k(U,V,X,Y) \ne 0}  } 
\sum_{\substack{q  \sim Q \\ q~\text{prime} \\ q \mid F_k(U,V,X,Y)}}  1.  
\end{align*}
Recall that $F_k$ is a polynomial with constant coefficients of degree $k^3$. 
Hence 
$F_k(U,V,X,Y) \ll N^{k^3}$, and thus trivially has at most $O\(\log N\)$ prime divisors. Hence, 
we derive 
$$
\sum_{\substack{ q  \sim Q \\ q~\text{prime}}}  \max_{j \in \F_q^*}   \sE_k(N;j,q) \ll        \frac{Q}{\log Q}   T_k(N)  + N^{4+o(1)} , 
$$
and applying Lemma~\ref{lem: Bound T(N)} we conclude the proof.

\begin{remark}
Furthermore it is easy to see that there is a constant $C>0$ such that if $N \le C q^{1/k^3}$ then 
$ F_k(n_1,n_2,n_3,n_4) \equiv 0 \bmod q$ with  $1 \le n_1,n_2,n_3, n_4 \le N$ implies 
$ F_k(n_1,n_2,n_3,n_4) = 0$. Hence in this range of $N$, using Lemma~\ref{lem: Bound T(N)}, 
we obtain $\sE_{k} (N;j,q)  \ll   N^2$ for every $q$.  
\end{remark}  
\section{Proof of Theorem~\ref{thm:E2k-Set}}

\subsection{Preliminary discussion} 
We need some facts about the {\it Gowers norms\/}, introduced in the celebrated  work of Gowers~\cite{Gow_4,Gow_m}  on  the first quantitative bound for the famous Szemer\'edi Theorem~\cite{Sz}  about sets avoiding arithmetic progressions of length four and longer. 
As an important step in the proof, 
Gowers~\cite{Gow_4,Gow_m} observes that there are very random sets having an unexpected number of arithmetic progressions of length $l\ge 4$.
An example is, basically, the set 
\begin{equation}\label{def:Gowers_Ak}
    \cA^{(k)} =\left \{ x \in \Z_N :~ x^k \in \{1,\ldots, c_k N\}  \right\} \,,
\end{equation} 
where $c_k >0$ is an appropriate constant, depending on $k\ge 2$ only (see the beginning of~\cite[Section~4]{Gow_m} and  also~\cite{Gow_exm}). 
Then the set $\cA^{(k)}$ has an enormous number of arithmetic progressions of length $k+2$ but the expected number of shorter progressions. 
In Theorem~\ref{thm:E2k-Set} we consider the sets $\cN^{1/k}$, where $\cN$ is a set with small doubling. Clearly, such sets generalise the construction~\eqref{def:Gowers_Ak}. 
Below we show 
that these sets are random in the sense, that they all have small  additive energy. Actually, we obtain a stronger property that Gowers norms of its characteristic functions are small and thus this has even more parallels to the Gowers 
construction~\eqref{def:Gowers_Ak}.  
On the other hand, sets $\cN^{1/k}$ preserve all essential combinatorial properties of the sets $\cA^{(k)}$. For example, for $k=2$ and any $s\neq 0$ we have for an arbitrary $x\in \cN^{1/2} \cap (\cN^{1/2} + s)$ 
$$
x^2 \in \cN \mand (x-s)^2 \in \cN.
$$
Thus $2sx - s^2 \in \cN-\cN $ or  $x\in (\cN-\cN + s^2)/2s$. Hence all intersections $\cN^{1/2} \cap (\cN^{1/2} + s)$ are additively rich sets exactly as in construction~\eqref{def:Gowers_Ak} (we literally use 
such facts in the proof of Theorem~\ref{thm:E2k-Set} below).

\subsection{Gowers norms}
Now we are ready to give general definitions. 
Suppose that  $G$ is an abelian group with the group operation $+$ and  $\cA\subseteq G$ is a finite set.
Having a sequence of elements  $s_1,\ldots,s_l \in G$ we define the set 
$$\cA_{s_1,\ldots,s_l} = \cA \cap (\cA -s_1) \cap \ldots \cap (\cA -s_l).
$$
Let 
$
\| \cA \|_{\mathcal{U}^{k}}
$
be the Gowers non-normalised $k$th-norm~\cite{Gow_m} of the characteristic function of $\cA$ (in additive form). We have, see, for example,~\cite{s_energy}:
$$
\| \cA \|_{\mathcal{U}^{k}}
=
\sum_{x_0,x_1,\ldots,x_k \in G}\,  \prod_{\varepsilon \in \{ 0,1 \}^k} \cA \left( x_0 + \sum_{j=1}^k \varepsilon_j x_j  \right) \,,
$$
where $\varepsilon = (\varepsilon_1,\ldots,\varepsilon_k)$ (we also recall that we use $\cA(a)$ for the indicator function of $\cA$). 
In particular, 
$$
\| \cA \|_{\mathcal{U}^{2}} = \sum_{x_0,x_1,x_2\in G} \cA(x_0) \cA(x_0 + x_1) \cA(x_0 + x_2) \cA(x_0 + x_1 + x_2) = E(\cA)
$$
is the additive energy of $\cA$, 
that is
$$
E(\cA) = \# \{(a_1,a_2,a_3,a_4) \in \cA^4:~a_1 + a_2 = a_3 + a_4\} ,
$$
 and
$$
\| \cA \|_{\mathcal{U}^{3}} = \sum_{s \in \cA-\cA} E(\cA_s) \,.
$$
Moreover, the induction property for Gowers norms holds, see~\cite{Gow_m}
\[
\| \cA \|_{\mathcal{U}^{k+1}} = \sum_{s \in \cA-\cA} \| \cA_s \|_{\mathcal{U}^{k}} 
\]
and 
\begin{equation}\label{f:Gowers_sums_A_s}
\| \cA \|_{\mathcal{U}^k} = \sum_{s_1,\ldots,s_k\in G} \# \cA_{\pi(s_1,\ldots,s_k)} \,,
\end{equation}  
where $\pi(s_1,\ldots,s_k)$ is a vector with $2^k$ components, namely,
$$
\pi(s_1,\ldots,s_k) = 
\left( \sum_{j=1}^k s_j \varepsilon_j \right)_{\(\varepsilon_1, \ldots, \varepsilon_k\) \in \{ 0,1 \}^k} \,.
$$ 
Notice also
\begin{equation}\label{f:Gowers_sq_A}
\| \cA \|_{\mathcal{U}^{k+1}} = \sum_{s_1,\ldots,s_k \in G} \(\# \cA_{\pi(s_1,\ldots,s_k)}\)^2 \,.
\end{equation}

It is proved in~\cite{Gow_m} that $k$th--norms of the characteristic function of any set are connected to each other.
It is shown in~\cite{s_energy}   that the connection for the non-normalised norms does not depend on size of the group $G$. 
Here we formulate a particular case of~\cite[Proposition 35]{s_energy}, 
which relates  $\| \cA \|_{\mathcal{U}^{k}}$ and $\| \cA \|_{\mathcal{U}^{2}}$.

\begin{lemma}
	\label{l:Gowers_char-1}
	Let $\cA$ be a finite subset of an abelian group $G$ with the group operation $+$.
	Then for any integer $k\ge 1$, we have
	$$
	    \| \cA \|_{\mathcal{U}^{k+1}} \ge
	    \frac{\| \cA \|^{(3k-2)/(k-1)}_{\mathcal{U}^{k}}}{\| \cA \|^{2k/(k-1)}_{\mathcal{U}^{k-1}}} \,.
	$$
\end{lemma}

Next we have to   relate  $\| \cA \|_{\mathcal{U}^{k}}$ and $E(\cA)$, see~\cite[Remark~36]{s_energy}.

 \begin{lemma}
	\label{l:Gowers_char-2}
	Let $\cA$ be a finite subset of an abelian group $G$ with the group operation $+$.
	Then for any integer $k\ge 1$, we have
		$$
	\| \cA \|_{\mathcal{U}^{k}} \ge E(\cA)^{2^k-k-1}\( \# \cA\)^{-(3\cdot 2^k -4k -4)} \,.
	$$
\end{lemma}

\subsection{Concluding the proof}	
Let $\cA = \cN^{1/k}$. 

\subsubsection{Case $k=3$} 
	Let us start with the case $k=3$. 
	Below we can assume that the quantity $L$ is sufficiently small because otherwise the result is  trivial. 
	
	For any $s\neq 0$ consider the set $\cA_s = \cA \cap (\cA-s)$ and let $x\in \cA_s$. 
	Then $x^3, (x+s)^3 \in \cN$ and hence 
	\[
		3s (x+s/2)^2  + s^3/4 = 3sx^2 + 3s^2 x + s^3 \in \cN - \cN \,.
	\] 
	Put $\cB_s =\cA_s + s/2$, so  $\# \cB_s = \# \cA_s$. Furthermore, let $\cC_s = \{x^2:~ x \in \cB_s\}$.
	Clearly, by the Pl\"unnecke inequality, see~\cite[Corollary~6.29]{TaoVu}, 
	\[
		\#(\cC_s + \cC_s) \le \#(2\cN - 2\cN) \le  L^4 N =  
		L_s \# \cA_s \,, 
	\]
	where 
	$$
	L_s =  \frac{L^4 N}{\# \cA_s}.
	$$
	Then, after applying   estimate~\eqref{eq:Energy-SSZ-Set-N} with our restriction $N\le q^{2/3}$, we  
	obtain 
  \begin{equation}\label{eq:EAs_EBs}
  \begin{split}
E(\cA_s) &   = 
E(\cB_s)  \ll    E_{2}(\cC_s;q)\\
&  \le  \(  L^4_s \(\# \cA_s\)^4/q+ L^2_s \(\# \cA_s\)^{11/4}\) q^{o(1)} \,.
\end{split}
\end{equation}

We now assume that 
\begin{equation}\label{eq: Large  A_s}
\#\cA_s \ge N^{4/5} L^{32/5}.
\end{equation}
We also observe that 
we can always assume that $L\le N^{1/32}$  as otherwise the result is trivial.  
Further to show that the second term in~\eqref{eq:EAs_EBs} dominates the first one,  
we need to check  that 
\begin{equation}\label{eq: T1<T2}
L^4_s \(\# \cA_s\)^4/q\le  L^2_s \(\# \cA_s\)^{11/4}
\end{equation}  
or $L^2_s \( \#\cA_s\)^{5/4} \le q$, which in turn is equivalent to $\( \#\cA_s\)^{3} \ge L^{32} N^8 q^{-4} $.
Since for  $L\le N^{1/32}$ and $N\le q^{2/3}$ we have 
$$
N^{12/5} L^{96/5}\ge L^{32} N^8 q^{-4} 
$$
we see that under the assumption~\eqref{eq: Large  A_s} we have~\eqref{eq: T1<T2} and hence the bound~\eqref{eq:EAs_EBs} 
becomes
  \begin{equation}\label{eq:EAs}
E(\cA_s)    \le   L^2_s \(\# \cA_s\)^{11/4}  q^{o(1)} \le   L^8 N^2\( \#\cA_s\)^{3/4} q^{o(1)} \,.
\end{equation}

    By the definition of the sets $\cA_s$, we have 
\begin{equation}\label{f:E_A_s}
    \sum_{s\in \cA-\cA} \# \cA_s =\( \# \cA\)^2 \,.
\end{equation}

Furthermore, using the definition of $\mathcal{U}_3$--norm we write
\begin{equation}\label{eq: AU3}
	\| \cA \|_{\mathcal{U}^3}   = \sum_{s \in \cA-\cA} E(\cA_s)  
=  \sum_{s :\, \# \cA_s \le  T}  E(\cA_s) +  \sum_{s :\, \# \cA_s > T}  E(\cA_s).
\end{equation}

First we observe that 
\begin{align*}
 \sum_{s :\, \# \cA_s \le  T}   E(\cA_s) &  =
 \#\{ (a_1, a_2,a_3, a_4, s) \in \cA^4\times\( \cA-\cA\) :\\
& \qquad \qquad \qquad a_1+a_2= a_3+a_4, \ \# \cA_s \le  T, \\  
& \qquad  \qquad \qquad \qquad \qquad a_i - s  \in \cA, \ i =1,   \ldots, 4\}\,.
\end{align*}
Thus,  for each of $E(\cA)$ choices of quadruples $(a_1, a_2,a_3, a_4) \in \cA^4$ with 
$a_1+a_2= a_3+a_4$,  there are at most $T$ possibilities for $s$ with 
 $ \# \cA_s \le  T$ and we derive
\begin{equation}\label{eq:small As}
 \sum_{s :\, \# \cA_s \le  T}   E(\cA_s)  \le T E(\cA) \,.
\end{equation}

We now choose 
\begin{equation}\label{eq:Def T}
T= 27 E(\cA)^{-4/5} L^{32/5} N^{16/5}
\end{equation}
and note that the trivial upper bound $E(\cA) \le (\# \cA)^3 \le 27N^3$ implies that 
$T \ge  N^{4/5} L^{32/5}$. Hence for any $s$ with  $\# \cA_s> T$
the condition~\eqref{eq: Large  A_s} is satisfied and so the bound~\eqref{eq:EAs} holds.

Hence, by  identity~\eqref{f:E_A_s}, we obtain 
\begin{equation}
\label{eq:large As}
\begin{split}
 \sum_{s :\, \# \cA_s > T}  E(\cA_s) & \le   L^8 N^2  q^{o(1)} \sum_{s :\, \# \cA_s > T}\( \#\cA_s\)^{3/4}\\
 & \le
  L^8 N^2 T^{-1/4}  q^{o(1)}\sum_{s :\, \# \cA_s> T}\#\cA_s \\
		& \le 
  L^8 N^2 \cdot N^{2} T^{-1/4}  q^{o(1)} =    L^8 N^4 T^{-1/4}	 q^{o(1)}\, .
\end{split} 
\end{equation}

The value of $T$ in~\eqref{eq:Def T} is chosen to balance the bounds~\eqref{eq:small As} and~\eqref{eq:large As} and thus from~\eqref{eq: AU3} we derive 
\[
	\| \cA \|_{\mathcal{U}^3} \le  E(\cA)^{1/5} L^{32/5} N^{16/5}  q^{o(1)} \,.
\]
	Finally, applying  Lemma~\ref{l:Gowers_char-2}, we obtain 
\[
	E(\cA) \le N^2 \| \cA \|^{1/4}_{\mathcal{U}^3} \le 
   L^{8/5} N^{14/5} E(\cA)^{1/20}  q^{o(1)} \,,
\]
and whence
\[
    E(\cA) \le 	 L^{32/19} N^{56/19} q^{o(1)}\,,
\]
which gives the desired result for $k=3$. 
 
\subsubsection{Case $k=4$} 
Next we  consider the case $k=4$. 
	Let 
	$$
	\cA_{\pi(s,t)} = \cA \cap (\cA-s) \cap (\cA-t) \cap (\cA-s-t)
	$$
	and let $x\in \cA_{\pi(s,t)}$.
	Then $x^4, (x+s)^4, (x+t)^4, (x+t+s)^4 \in \cN$ and hence $\cN - \cN$ contains 
\[
	4u x^3 + 6 u^2 x^2 + 4   u^3 x+ u^4, \qquad u \in \{s,t, s+t\}.
\]
	Subtracting  the expressions with $s$ and $t$ from the expression with $s+t$, we see that  $3\cN-3\cN$ 
	contains $12 st x^2 + 12 (t^2 s + ts^2) x + (t+s)^4-s^4-t^4$ and we can apply   a version of previous 
	arguments. 
	Actually, in our particular  case $k=4$ one can write exact identity
\[
	(x+t+s)^4 + x^4 - (x+s)^4 - (x+t)^4 = 12 st x^2 + 12 (t^2 s + ts^2) x + (t+s)^4-s^4-t^4 
\]
	and thus even it is enough to consider the set $2\cN-2\cN$. 
	 In particular, since by the Pl\"unnecke inequality, see~\cite[Corollary~6.29]{TaoVu}, 
	\[
      \#(2\cN - 2\cN) \le  L^4 N  
     \]
 the role of $L_s$ is now played by
	$$
	L_{s,t}  =  \frac{L^{8} N}{\# \cA_{\pi(s,t)}}.
	$$
	We also  set 
	$$
	T =  (E(\cA) N^{2} L^{16} \| \cA \|_{\mathcal{U}^3}^{-1})^{4/5}
	$$
	and note that  we have the trivial bound  $\| \cA \|_{\mathcal{U}^3} \le N E(\cA)$.
	We also have 
	$$
	T \ge   N^{4/5} L^{64/5}.
	$$
	We now verify that $T^3 \ge  L^{64} N^8 q^{-4}$ or 
	$$
	 N^{12/5} L^{192/5} \ge L^{64} N^8 q^{-4}
	 $$
	 which is equivalent to $N^{28} L^{128} \le q^{20} $.
	 Since we can clearly assume that $L\le N^{1/64}$ as otherwise the result is
	 trivial,  the last inequality hold under our assumption   
	$N\le q^{2/3}$.
		
	Hence, similar to the case $k=3$ after simple calculations,  one verifies that  for 
	$\#\cA_{s,t} > T$,  we have $	L_{s,t}^2 \(\#\cA_{\pi(s,t)}\)^{5/4}\le q$ which in turn 
	is equivalent to 	
	$$
	\(\#\cA_{\pi(s,t)}\)^3 \ge T^3 \ge  L^{64} N^8 q^{-4}. 
	$$ 
	Therefore, by~\eqref{eq:Energy-SSZ-Set-N}   we have
\begin{align*}
E(\cA_{\pi(s,t)}) & \le  \(  L^4_{s,t} \(\# \cA_{\pi(s,t)}\)^4/q+ L^2_{s,t} \(\# \cA_{s,t}\)^{11/4}\) q^{o(1)}\\
 & \le L^2_{s,t} \(\# \cA_{s,t}\)^{11/4} q^{o(1)}\\
    & =  L^{16} N^2 \(\#\cA_{\pi(s,t)}\)^{3/4} q^{o(1)}.
\end{align*} 
	Using~\eqref{f:Gowers_sums_A_s}  and~\eqref{f:Gowers_sq_A}  and the arguments as above, we get 
\begin{equation}
\label{eq:AU4-Up}
\begin{split}
	\| \cA \|_{\mathcal{U}^4} &= \sum_{s,t} E(\cA_{\pi(s,t)}) \\
	& \le T \| \cA \|_{\mathcal{U}^3} +  L^{16} N^2  q^{o(1)} \sum_{(s,t) :\, \#\cA_{\pi(s,t)} > T} \#(\cA_{\pi(s,t)})^{3/4} 
	\\ &\le
T \| \cA \|_{\mathcal{U}^3} +   L^{16} N^2 E(\cA) T^{-1/4}  q^{o(1)}\\
&	\le 
	  L^{64/5}  N^{8/5} E^{4/5} (\cA) \| \cA \|^{1/5}_{\mathcal{U}^3}  q^{o(1)}\,
\end{split} 
\end{equation}
since again we have chosen $T$ to optimise the above bound. 

On the other hand,  applying  Lemma~\ref{l:Gowers_char-1} 
and then  Lemma~\ref{l:Gowers_char-2}, we derive 
\begin{equation}
\label{eq:AU4-Low}
     \| \cA \|_{\mathcal{U}^4} \ge   \frac{\| \cA \|^{7/2}_{\mathcal{U}^3}}{\| \cA \|_{\mathcal{U}^2}^3}    
     =   \frac{\| \cA \|^{7/2}_{\mathcal{U}^3}}{E^3(\cA)}   \ge 
\| \cA \|^{1/5}_{\mathcal{U}^3} \cdot \frac{E^{51/5}(\cA)}{N^{132/5}}.  
\end{equation}
Comparing~\eqref{eq:AU4-Up} and~\eqref{eq:AU4-Low}
\[
	E(\cA) \le L^{64/47} N^{3-1/47} q^{o(1)} \,,
\]
which gives the desired result for $k=4$. 

\subsubsection{Case $k\ge 5$} 
	Finally, consider the general case, which we treat with a version of {\it Weyl differencing\/}.
	Now 
$$
\cA_{\vec{s}} = 
\cA_{\pi(s_1,\ldots,s_{k-2})}
$$ and let $x\in \cA_{\pi(s_1,\ldots,s_{k-2})}$. 
	Indeed, we start with $\cA_{s_1}$ and reduce the main term in $x^k, (x+s_1)^k \in \cN$ deriving that
	$p_{k-1} (x) \in \cN-\cN$, where $\deg p_{k-1}= k-1$. 	After that consider $(\cA_{s_1})_{s_2} = \cA_{\pi(s_1,s_2)}$ and reduce  degree of the polynomial by one, 
	and so on. 
	 We also note that  by the Pl\"unnecke inequality, see~\cite[Corollary~6.29]{TaoVu}, 
	\[
      \#\(2^{k-1}\cN - 2^{k-1}\cN\) \le  L^{2^k} N    
     \]
 the role of $L_s$ or $L_{s,t}$  is now played by
	$$
	L_{\vec{s}}  =  \frac{L^{2^{k}} N}{\# \cA_{\pi(\vec{s})}}.
	$$
	
	We  now set 	 
	$$
	T = \(N^{2} L^{2^{k+1}} \| \cA \|_{\mathcal{U}^{k-2}} \| \cA \|_{\mathcal{U}^{k-1}}^{-1}\)^{4/5}.
	$$
	Using the same arguments as  above, after somewhat tedious calculations to verify all 
	necessary conditions such as  
\begin{equation}\label{tmp:L_s&T}
    N^8 
    L^{2^{k+3}} 
    q^{-4} \le \(\#\cA_{\pi(s_1,\ldots,s_{k-2})}\)^3
\end{equation}
	to obtain 
	$$
	E(\cA_{\pi(s_1,\ldots,s_{k-2})}) \le  L^{2^{k+1}} N^2 \(\#\cA_{\pi(s_1,\ldots,s_{k-2})}\)^{3/4} q^{o(1)}.
	$$
	In particular,  to check~\eqref{tmp:L_s&T} we note that for the above choice of $T$ we have 
	$$
	T \ge N^{4/5}L^{2^{k+3}/5}, 
	$$ 
	and then derive  
\[
    N^8 L^{2^{k+3}} q^{-4} \le N^{12/5} L^{3\cdot  2^{k+3}/5} \le T^3 
\]
which is true because $N\le q^{2/3}$ and $L\le N^{1/ 2^{k+3}}$ (which we can assume as 
otherwise the bound is trivial). 

	Using the formula~\eqref{f:Gowers_sums_A_s} and~\eqref{f:Gowers_sq_A}
	we obtain  
\begin{align*}
\| \cA \|_{\mathcal{U}^k} & \le T \| \cA \|_{\mathcal{U}^{k-1}} +  L^{2^{k+1}} N^2 q^{o(1)} \sum_{\vec{s} :\, \#\cA_{\pi(\vec{s})} > T} \#(\cA_{\pi(\vec{s})})^{3/4} \\
& \le
T \| \cA \|_{\mathcal{U}^{k-1}} + L^{2^{k+1}} N^{2} \| \cA\|_{\mathcal{U}^{k-2}} T^{-1/4} q^{o(1)}\\
& \le
 L^{2^{k+1} \cdot 4/5}  N^{8/5} \| \cA \|^{4/5}_{\mathcal{U}^{k-2}} \| \cA \|^{1/5}_{\mathcal{U}^{k-1}} q^{o(1)}
\end{align*} 
and 
hence by induction and Lemma~\ref{l:Gowers_char-2}
\[
E(\cA)^{7\cdot 2^{k-1}-9} \le  L^{2^{k+3}} N^{21\cdot 2^{k-1}-28}  q^{o(1)}.
\]
In other words, 
\[
E(\cA) \le   L^{2^{k+3}/(7\cdot 2^{k-1}-9)}  N^{3-1/(7\cdot 2^{k-1}-9)}q^{o(1)}\,, 
\]
which completes the proof.

\section{Proof of Theorem~\ref{thm:b1}}
Given a function $f:\F_q\rightarrow \C$, we define the Fourier transform of $f$ by
$$
\widehat f(n)=\frac{1}{q^{1/2}}\sum_{\lambda \in \Fq}f(\lambda)\eq(\lambda n).
$$
Define 
\begin{equation}
\label{eq:fmn}
f_m(n)=\sum_{\substack{x\in \Fq \\ x^2=amn}}\eq(hx),
\end{equation}
so that 
$$
V_{a,q}(\balpha, \varphi;h,M,N)=\sum_{m\sim M}\alpha_m\sum_{n\in \Z}\varphi(n)f_m(n).
$$
Recall that $\varphi$ satisfies~\eqref{eq:cond phi}. 

Applying Poisson summation to  the sum  over  $n$ gives 
\begin{equation}
\label{eq:Vstep1}
V_{a,q}(\balpha, \varphi;h,M,N)= \frac{N}{q^{1/2}}\sum_{m\sim M}\alpha_m\sum_{n\in \Z}\widehat \varphi\left(-\frac{n}{q}\right)\widehat f_m(n),
\end{equation}
where 
$$
\widehat f_m(n)=\frac{1}{q^{1/2}}\sum_{\lambda \in \Fq}f_m(\lambda)\eq(\lambda n).
$$
Using~\eqref{eq:fmn} and interchanging summation
\begin{align*}
\widehat f_m(n)&=\frac{1}{q^{1/2}}\sum_{x \in \Fq}\sum_{\substack{\lambda\in \Fq \\ x^2=am\lambda}}\eq(hx)\eq(\lambda n) \\ 
&=\frac{1}{q^{1/2}}\sum_{x \in \Fq} \eq(\overline{am}n x^2+hx),
\end{align*}
where $\overline{am}$ denotes multiplicative inverse modulo $q$. Summation over $x$ is a quadratic Gauss sum which has evaluation, see~\cite[Theorem~1.52]{BEW}
$$
\widehat f_m(n)=\varepsilon_q \chi(amn)\eq(-am\overline {4n}h^2),
$$
for some $|\varepsilon_q|=1$, where $\chi$ is the quadratic character mod $q$. Therefore, there exists some integer $c$ with $\gcd(c,q)=1$ depending on $a$ and $h$ such that 
$$
\widehat f_m(n)=\varepsilon_q \chi(amn)\eq(cm\overline {n}).
$$
Substituting this into~\eqref{eq:Vstep1} and applying the triangle inequality, we obtain 
$$
\left|V_{a,q}(\balpha, \varphi;h,M,N)\right| \ll \frac{1}{q^{1/2}}\sum_{m\sim M}\left|\sum_{n\in \Z}\widehat \varphi\left(-\frac{n}{q}\right)\chi(n)\eq(cm\overline {n})\right|.
$$
Our next step is to apply linear shifts in a similar fashion to Friedlander and Iwaniec's generalisation of the Burgess bound for character sums~\cite{FI}. Define 
\begin{equation}
\label{eq:Vdef}
U=\frac{q}{MN},
\end{equation}
so by assumption on $M,N$ we have  $U\gg 1$. For fixed $m\sim M$ apply shifts $n\rightarrow n+um$ to the inner summation over $n$. Averaging this over  $1\le u \le U$ gives
\begin{align*}
& V_{a,q}(\balpha, \varphi;h,M,N) \\
& \qquad \ll \frac{1}{q^{1/2}U}\sum_{m\sim M}\sum_{n\in \Z}\\
& \qquad \qquad \qquad \quad  \left|\sum_{1\le u \le U}\widehat \varphi\left(-\frac{n+mu}{q}\right)\chi(n+mu)\eq(cm(\overline {n+mu}))\right|.
\end{align*} 
Let $\varepsilon>0$ be small. Note by~\eqref{eq:cond phi}  and partial integration, for any $m\sim M$, $1\le u \le U$  and constant $C>0$ we have 
$$
\widehat \varphi\left(-\frac{n+mu}{q}\right)\ll \frac{1}{n^{C}}, \quad \text{provided} \quad n\ge \frac{q^{1+\varepsilon}}{N}.
$$
Therefore 
\begin{align*}
& V_{a,q}(\balpha, \varphi;h,M,N)\\
& \qquad \ll \frac{1}{q^{1/2}U} \sum_{m\sim M}\sum_{|n|\le q^{1+\varepsilon}/N} \\
& \qquad \qquad \qquad \quad  \left|\sum_{1\le u \le U}\widehat \varphi\left(-\frac{n+mu}{q}\right)\chi(n+mu)\eq(cm(\overline {n+mu}))\right|.
\end{align*} 
Applying partial summation to $u$ and using 
$$
\frac{\partial \widehat\varphi\left(-\frac{n+mu}{q}\right)}{\partial u}\ll \frac{N}{|u|},
$$
we obtain
\begin{align*}
V_{a,q}(\balpha, \varphi;h,M,N)& \ll \frac{N^{1+o(1)}}{q^{1/2}U}  \sum_{m\sim M}
\\&\qquad \quad  \sum_{\substack{|n|\le q^{1+\varepsilon}/N}}\left|\sum_{1\le u\le U_0}\chi(n\overline{m}+u)\eq(c\overline{(n\overline{m}+u)})\right|,
\end{align*}
for some $U_0\le U$. Let $I(\lambda)$ count the number of solutions to 
$$
\lambda \equiv nm^{-1} \bmod{q}, \quad  |n|\le \frac{q^{1+o(1)}}{N}, \quad m \sim M,
$$
so that 
\begin{equation}
\label{eq:Vaq-I}
\begin{split}
& V_{a,q}(\balpha, \varphi;h,M,N)\\
& \qquad \quad \le \frac{N^{1+o(1)}}{q^{1/2}U}\sum_{\lambda \in \Fq}I(\lambda)\left|\sum_{1\le u\le U_0}\chi(\lambda+u)\eq(c\overline{(\lambda+u)})\right|.
\end{split} 
\end{equation}
Note 
\begin{equation}
\label{eq:lambdaell1}
\sum_{\lambda \in \Fq}I(\lambda)\ll \frac{q^{1+\varepsilon}M}{N},
\end{equation}
and 
\begin{align*}
\sum_{\lambda \in \Fq}I(\lambda)^2 = 
\#\{(m_1,m_2, n_1 , n_2)&\in \Z^4:~n_1m_2 \equiv n_2m_2 \bmod{q},\\
 &    |n_1|, |n_2| \le \frac{q^{1+\varepsilon}}{N}, \ m_1,m_2 \sim M\}.
\end{align*}
It is known (see, for example,~\cite{ACZ}) that 
$$
\sum_{\lambda \in \Fq}I(\lambda)^2\le q^{2\varepsilon+o(1)}\left(\frac{1}{q}\left(\frac{q M}{N}\right)^2+\frac{qM}{N}+M^2\right),    
$$
and by assumptions on $M,N$ the above simplifies to 
\begin{equation}
\label{eq:lambdaell2}
\sum_{\lambda \in \Fq}I(\lambda)^2\ll \frac{q^{1+2\varepsilon}M}{N}.
\end{equation} 
Applying the H\"{o}lder inequality to summation in~\eqref{eq:Vaq-I} gives 
\begin{align*}
V_{a,q}(\balpha, \varphi;h,M,N)^{2r}&\ll \frac{N^{2r+o(1)}}{q^{r}U^{2r}}\left(\sum_{\lambda\in \Fq}I(\lambda) \right)^{2r-2}\left(\sum_{\lambda\in \Fq}I(\lambda)^2 \right) \\ 
& \qquad  \qquad \times \sum_{\lambda \in \Fq}\left|\sum_{1\le u\le U_0}\chi(\lambda+u)\eq(c\overline{(\lambda+u)})\right|^{2r}.
\end{align*}
Using~\eqref{eq:lambdaell1} and~\eqref{eq:lambdaell2}
\begin{align*}
& V_{a,q}(\balpha, \varphi;h,M,N)^{2r}\\
& \quad \quad \le q^{r-1+4r\varepsilon+o(1)}NM^{2r-1}\frac{1}{U^{2r}}\sum_{\lambda \in \Fq}\left|\sum_{1\le u\le U_0}\chi(\lambda+u)\eq(c\overline{(\lambda+u)})\right|^{2r}.
\end{align*}
Expanding the $2r$-th power, interchanging summation, isolating the diagonal contribution and using the Weil bound (see~\cite[pg. 45, Theorem~2G]{Schm}) gives 
$$
\sum_{\lambda \in \Fq}\left|\sum_{1\le u\le U_0}\chi(\lambda+u)\e_q(c\overline{(\lambda+u)})\right|^{2r}\ll q^{1/2}U^{2r}+qU^{r}.
$$
Using in the above and recalling~\eqref{eq:Vdef}, we get 
\begin{align*}
V_{a,q}(\balpha, \varphi;h,M,N)^{2r}&\ll q^{r-1+4r\varepsilon+o(1)}NM^{2r-1}\left(q^{1/2}+\frac{q}{U^r}\right) \\
&\ll q^{r-1/2+4r\varepsilon+o(1)}NM^{2r-1}\left(1+\frac{(MN)^{r}}{q^{r-1/2}}\right),
\end{align*}
from which the result follows after taking $\varepsilon$ sufficiently small.

\section{Proof of Theorem~\ref{thm:GammaqP}}

\subsection{Preliminaries} 
Our argument follows the proof of~\cite[Theorem~1.10]{DKSZ}, the only difference being our use of Corollary~\ref{cor:Waq}  and Theorem~\ref{thm:b1}. We refer the reader to~\cite[Section~7]{DKSZ} for more complete details.   

Let $\widetilde{S}_q(h,P)$ denote the sum
$$
\widetilde{S}_{q}(h,P) = \sum_{k=1}^P \Lambda(k) \sum_{\substack{x \in \Fq \\ x^2 =k}} \eq(hx).
$$
By partial summation, it is sufficient to show 
$$
\widetilde{S}_{q}(h,P)\le  \(q^{1/16}P^{7/8}+q^{13/88}P^{3/4}\)q^{o(1)}.
$$
Let $J\ge 1$ be an integer. Using the Heath-Brown identity and a smooth partition of unity as in~\cite[Section~1.7]{DKSZ}, there exists some 
$$
\mathbf{V}=(M_1,\ldots , M_J,N_1,\ldots  ,N_J) \in [1/2,2P]^{2J} 
$$
 $2J$-tuple of parameters satisfying
$$
N_1 \geq \ldots \ge N_J, \quad  M_1,\ldots ,M_J \leq P^{1/J},\quad   P \ll  Q \ll  P,
$$
(implied constants are allowed to depend on $J$),  
 \begin{equation} \label{eq:prod Q}
Q =  \prod_{i=1}^J M_i \prod_{j =1}^JN_j,
\end{equation}  
and
\begin{itemize}
\item the arithmetic functions $m_i \mapsto \gamma_i(m_i)$ are bounded and supported in $[M_i/2,2M_i]$;
\item the smooth functions $x_i \mapsto V_i(x)$ have support in $[1/2,2]$ and satisfy
$$
V_i^{(j)}(x) \ll  q^{j \varepsilon}
$$
for all integers $j \geq 0$, where the implied constant may depend on $j$ and $\varepsilon$. 
\end{itemize}  
such that defining 
\begin{align*}
\Sigma(\mathbf{V})=\sum_{m_1, \ldots, m_J=1}^{\infty} &\gamma_1(m_1)\cdots  \gamma_J(m_J)  \sum_{n_1,\ldots , n_J=1}^\infty \\
\\& V_1 \( \frac{n_1}{N_1} \)  \cdots V_J \( \frac{n_J}{N_J} \)
\sum_{\substack{x \in \Fq \\ x^2=m_1 \cdots m_J n_1 \cdots  n_J}} \eq(hx), 
\end{align*}  
we have 
$$
\widetilde{S}_{q}(h,P)\ll P^{o(1)}\Sigma(\mathbf{V}).
$$
We proceed on a case by case basis depending on the size of $N_1$.  We first note a general estimate for the multilinear sums. Let $\cI,\cJ\subseteq \{1,\ldots,J\}$ and write 
$$
M=\prod_{i\in \cI}M_i\prod_{j\in \cJ}N_j, \quad N=Q/M.
$$
Grouping variables in $\Sigma(\mathbf{V})$ according to $\cI,\cJ$, there exists $\alpha,\beta$ satisfying 
$$
\|\alpha\|_{\infty}, \|\beta\|_{\infty}=Q^{o(1)},
$$ 
such that 
$$
\Sigma(\mathbf{V})=\sum_{\substack{m\le 2^J M \\ n\le 2^J N}}\alpha(m)\beta(n)\sum_{\substack{x\in \mathbb{F}_q \\ x^2=mn}}\eq(hx).
$$
By Corollary~\ref{cor:Waq}
\begin{equation}
\begin{split} 
\label{eq:Swigma4}
& \Sigma(\mathbf{V})\\
&\quad \le  q^{1/8+o(1)}P^{3/4}\left(\frac{P^{3/16}}{q^{1/16}M^{3/16}}+1\right)\left(\frac{M^{3/16}}{q^{1/16}}+1\right) \\
&\quad \le q^{o(1)}\left(P^{15/16}+\frac{q^{1/16}P^{15/16}}{M^{3/16}}+q^{1/16}P^{3/4}M^{3/16}+q^{1/8}P^{3/4} \right).
\end{split} 
\end{equation}
We proceed on a case by case basis depending on the size of $N_1$. Let $P^{1/2}\ge H\ge P^{\varepsilon}$ be some paramters and take  
$$
J = \rf{\log P/\log H}.
$$

\subsection{Small $N_1$}
Suppose first $N_1\le H$ then arguing as in~\cite[Equation~(7.13)]{DKSZ} we
can  choose two arbitrary sets $\cI, \cJ \subseteq\{1, \ldots, J\}$ such that for 
$$
M = \prod_{i\in \cI} M_i \prod_{j \in \cJ} N_j \mand N = Q/M, 
$$
where $Q$ is given by~\eqref{eq:prod Q} and we have 
$$
P^{1/2} \ll M \ll H^{1/2}P^{1/2}. 
$$
Hence by~\eqref{eq:Swigma4}
\begin{equation}
\label{eq:case1}
\Sigma(\mathbf{V})\le q^{o(1)}\left(P^{15/16}+q^{1/16}P^{27/32}H^{3/32}+q^{1/8}P^{3/4} \right).
\end{equation}
\subsection{Medium $N_1$}
Let $L$ be a parameter satisfying $H\le L$ and suppose next that 
$$
H\le N_1\le L.
$$
We may also suppose 
$$H\le N_2\le N_1\le L,$$
as otherwise we may argue as before to obtain the bound~\eqref{eq:case1}. In this case we define $M,N$ as 
$$
N=\prod_{i=1}^{J}M_i\prod_{j=3}^{J}N_j \quad \text{and} \quad M=N_1N_2,
$$
so that 
$$
H^2\le M\le L ^2.
$$
By~\eqref{eq:Swigma4}
\begin{equation}
\begin{split}
\label{eq:case2}
\Sigma(\mathbf{V})\le q^{o(1)}\biggl(P^{15/16}& +\frac{q^{1/16}P^{15/16}}{H^{3/8}}\\
& \quad +q^{1/16}P^{3/4}L^{3/8}+q^{1/8}P^{3/4} \biggr).
\end{split} 
\end{equation}
\subsection{Large $N_1$}
Let $R$ be a parameter to be chosen later and satisfying $R\ge c P^{1/2}$ for some sufficiently 
large constant $c > 0$. 
Suppose next that 
$$L\le N_1\le R.$$
Taking $M=N_1$ as above, we derive from~\eqref{eq:Swigma4}
\begin{equation}
\begin{split} 
\label{eq:case3}
\Sigma(\mathbf{V} ) \le   q^{o(1)}\biggl(P^{15/16}& +\frac{q^{1/16}P^{15/16}}{L^{3/16}}\\
&\quad  +q^{1/16}P^{3/4}R^{3/16}+q^{1/8}P^{3/4} \biggr).
\end{split} 
\end{equation}

\subsection{Very large $N_1$}
Finally consider when $N_1\ge R$. We now intend to apply Theorem~\ref{thm:b1} with 
$P/N_1\ll M\ll P/N_1$ and $N = N_1$, where we notice that 
the condition  $R\ge c P^{1/2}$ ensures that $M< N$, provided that $c$ is large enough. 
Choosing $r=2$ we obtain
\begin{align*}
\Sigma(\mathbf{V}) & \le  q^{3/8+o(1)}(P/N_1)^{3/4} N_1^{1/4} \(1+ \frac{P^{1/2}}{q^{3/8}}\)\\
& = q^{3/8+o(1)}P^{3/4} N_1^{-1/2} \(1+ \frac{P^{1/2}}{q^{3/8}}\). 
\end{align*}
Using the assumption $P\le q^{3/4}$  we obtain
\begin{equation}
\label{eq:case4}
\Sigma(\mathbf{V}) \le  q^{3/8+o(1)} \frac{P^{3/4}}{R^{1/2}}.
\end{equation}

\subsection{Optimisation}
Combining all previous bounds~\eqref{eq:case1}, \eqref{eq:case2}, \eqref{eq:case3} and~\eqref{eq:case4} results in
\begin{align*}
\widetilde{S}_{q}(h,P)&\le q^{o(1)}(P^{15/16}+q^{1/8}P^{3/4})\\
&  \qquad \qquad+q^{o(1)}\left(q^{1/16}P^{27/32}H^{3/32}+\frac{q^{1/16}P^{15/16}}{H^{3/8}}\right) \\ 
&  \qquad \qquad \qquad  +q^{o(1)}\left( q^{1/16}P^{3/4}L^{3/8}+\frac{q^{1/16}P^{15/16}}{L^{3/16}}\right) \\ &  \qquad \qquad \qquad \qquad +q^{o(1)}\left(q^{1/16}P^{3/4}R^{3/16}+q^{3/8+o(1)}\frac{P^{3/4}}{R^{1/2}}\right).
\end{align*}
Taking parameters 
$$
H=P^{1/5}, \quad L=P^{1/3}, \quad R=q^{5/11},
$$
gives
$$
\widetilde{S}_{q}(h,P)\le q^{o(1)}\(P^{15/16}+q^{1/8}P^{3/4}+q^{1/16}P^{7/8}+q^{13/88}P^{3/4}\).
$$
We now observe that for $P < q^{1/2}$ the bound is trivial, while for $P\ge q^{1/2}$ we have
$$
q^{1/16} P^{7/8} \ge q^{1/8}P^{3/4}.
$$
Furthermore, under our assumption $P \le q^{3/4}$ we have 
$$
P^{15/16} \le q^{13/88}P^{3/4}
$$
and thus the above bound becomes 
$$
\widetilde{S}_{q}(h,P)\le  q^{o(1)}\(q^{1/16}P^{7/8}+q^{13/88}P^{3/4}\),
$$
which completes the proof.

\section*{Acknowledgement} 

We would like to thank \begin{itemize}
\item Christian Bagshaw for pointing out a gap in the initial proof of~\eqref{eq:S1} 
and his help with fixing it; 
\item  Alexander Dunn for some useful discussions and in particular, for pointing out the paper of Duke~\cite{Duke} regarding multidimensional Sali\'{e} sums;
\item Yixiu Xiao for attracting our attention to a gap in placing $N_1$  in the various ranges in the initial proof of Theorem~\ref{thm:GammaqP}. 
\end{itemize}

We are also very grateful to the referee for 
the very careful reading of the manuscript and very helpful comments.

During the preparation of this work,  B.K. was supported  by the   Academy of Finland Grant~319180 and by the Max Planck Institute for Mathematics, 
I.D.S.  by the Ministry of Science and Higher Education of the Russian Federation (agreement no. 075-02-2023-934), 
and 
 I.E.S.   by the Australian Research Council Grants DP170100786 and DP200100355.

\end{document}